\documentclass[11pt,letterpaper]{amsart}

\usepackage{amsmath}

\usepackage{float}
\usepackage{placeins,needspace}
\usepackage{hyperref,color,amssymb}
\usepackage{graphicx}
\usepackage{verbatim}
\usepackage{comment}
\usepackage{xcolor}
 
\newif\ifrevfinal \revfinaltrue
\DeclareRobustCommand{\rev}[1]{\ifrevfinal #1\else\textcolor{blue}{#1}\fi}

\hypersetup{hidelinks,pdftitle={On the Dirichlet boundary value problem on Cartan-Hadamard manifolds},pdfauthor={Marcos P. Cavalcante, Jose M. Espinar, Diego A. Marin}}
\newtheorem{theorem}{Theorem}[section]
\newtheorem{lemma}[theorem]{Lemma}
\newtheorem{proposition}[theorem]{Proposition}
\newtheorem{corollary}[theorem]{Corollary}

\newtheorem{definition}[theorem]{Definition}

\theoremstyle{remark}
\newtheorem{remark}[theorem]{Remark}

\numberwithin{equation}{section}

\newcommand{\II}{\mathrm I\!\mathrm I}
\newcommand{\dist}{\operatorname{dist}}

\begin{document}

\title
[On the Dirichlet boundary value problem]
{On the Dirichlet boundary value problem on Cartan-Hadamard manifolds}

\author[Cavalcante]{Marcos P. Cavalcante}
\address{Institute of Mathematics, Federal University of Alagoas, Macei\'o, Brazil}
\email{marcos.cavalcante@im.ufal.br}
\author[Espinar]{Jos\'e M. Espinar}
\address{Department of Geometry and Topology and Institute of Mathematics (IMAG), University of Granada, 18071 Granada, Spain}
\email{jespinar@ugr.es}
\thanks{Corresponding author: Jos\'e M. Espinar, \texttt{jespinar@ugr.es}.}
\author[Mar\'in]{Diego A. Mar\'in}
\address{Department of Geometry and Topology and Institute of Mathematics (IMAG), University of Granada, 18071 Granada, Spain}
\email{damarin@ugr.es}

\makeatletter
\@ifundefined{subjclassname@2020}{\@namedef{subjclassname@2020}{\textup{2020} Mathematics Subject Classification}}{}
\makeatother

\date{October 23, 2026}


\begin{abstract}
We study bounded solutions of Dirichlet boundary value problems for semilinear elliptic equations on Cartan--Hadamard manifolds with pinched negative curvature. We prove nonexistence when the asymptotic boundary has zero Hausdorff measure at an exponent determined by the upper curvature bound and the spectral parameter. The result holds for continuous viscosity solutions under a one-sided bound on the nonlinearity, including the Allen--Cahn equation. The proof combines convex barriers at infinity, distance comparison for $C^{1,1}$ hypersurfaces and a spectral maximum principle. For surfaces, no lower curvature bound is needed.
\end{abstract}

\maketitle

\section{Introduction}\label{intro}

The study of bounded solutions to semilinear elliptic equations on Riemannian manifolds is a central topic in geometric analysis, lying at the interface of differential geometry, potential theory, and partial differential equations. 

In contrast with the Liouville theorem for bounded harmonic functions on complete manifolds with nonnegative Ricci curvature \cite{Yau75}, a Cartan--Hadamard manifold with pinched negative curvature admits nonconstant bounded harmonic functions.

Recall that a Cartan--Hadamard manifold is complete, simply connected and has nonpositive sectional curvature. In this setting, the manifold admits a compactification $\bar{M}^n=M^n\cup\partial_\infty M^n$, where $\partial_\infty M^n$ is the ideal boundary. For a domain $\Omega\subset M^n$, its asymptotic boundary $\partial_\infty\Omega$ consists of the ideal points in its closure in the cone topology. This distinguishes it from the ordinary boundary $\partial\Omega\subset M^n$.

When the sectional curvature is pinched between two negative constants, Anderson~\cite{Anderson83} and Sullivan~\cite{Sullivan83} proved that every continuous function on $\partial_\infty M^n$ admits a continuous harmonic extension to $\bar{M}^n$. This result was refined by Anderson and Schoen~\cite{AndersonSchoen85}, who provided explicit representation formulas for such harmonic functions (see also~\cite[Chapter~2]{SY94}). Ancona~\cite{Ancona87} further extended these results to a broader class of elliptic operators using Martin boundary theory. These foundational works show that Cartan--Hadamard manifolds with pinched negative curvature support a large class of bounded harmonic functions.

These developments naturally lead to the study of $\lambda$-harmonic functions, i.e., solutions to $\Delta u+\lambda u=0$, where $\Delta$ is the Laplace-Beltrami operator. We use the convention that $\lambda_0(M^n)$ is the bottom of the spectrum of $-\Delta$, namely
\[
\lambda_0(M^n)=\inf_{0\ne\psi\in C_c^\infty(M^n)}
\frac{\int_{M^n}|\nabla\psi|^2}{\int_{M^n}\psi^2}.
\]

On $M^n$, the equation above has a positive solution precisely when $\lambda\le\lambda_0(M^n)$; see Sullivan~\cite[Theorem~2.1]{Sullivan87}. Such a solution need not be bounded.

A remarkable connection between the spectral theory of negatively curved manifolds and the geometry at infinity was established by Sullivan~\cite[Theorem~2.21]{Sullivan87}. Let $M^n=\mathbb H^n/\Gamma$ be a geometrically finite hyperbolic manifold, and denote the limit set of $\Gamma$ by $\Lambda(\Gamma)$. If its Hausdorff dimension satisfies $\dim_H\Lambda(\Gamma)\le(n-1)/2$, then $\lambda_0(M^n)=\lambda_0(\mathbb H^n)=(n-1)^2/4$. This relates the spectrum to the size of a set at infinity, measured through Hausdorff dimension.

Here we consider a different boundary value problem. Let $\Omega\subset M^n$ be a domain, and prescribe zero data on its boundary $\partial\Omega\subset M^n$. We ask how small its asymptotic boundary $\partial_\infty\Omega$ must be in order to rule out nontrivial bounded solutions of
\[
\Delta u+f(u)=0\quad\hbox{in }\Omega,
\qquad u=0\quad\hbox{on }\partial\Omega.
\]

Bonorino and Klaser~\cite{BonorinoKlaser} investigated the existence of bounded $\lambda$-harmonic functions, $f(t)=\lambda t$, in unbounded domains $\Omega\subset\mathbb H^n$. They proved that the condition $\mathcal H^{(n-1)/2}(\partial_\infty\Omega)=0$ excludes nontrivial bounded solutions of the Dirichlet boundary value problem for $0\le\lambda\le(n-1)^2/4$. Their examples at $\lambda=(n-1)^2/4$ show that this uniform threshold is sharp.

On the other hand, the second author and Mao~\cite[Lemma~2.2 and Theorem~2.7]{EspinarMao} studied geometric restrictions on domains carrying positive solutions of $\Delta u+f(u)=0$ under the opposite type of bound. On a Cartan--Hadamard manifold with $-\kappa_1\le K\le-\kappa_2<0$, they assume $f(t)\ge\lambda t$ for $t>0$, with $\lambda>(n-1)^2\kappa_1/4$. Their results exclude sufficiently large geodesic balls and conical neighborhoods at infinity.

In this paper, we extend the nonexistence argument to pinched Cartan--Hadamard manifolds and to continuous nonlinearities satisfying a one-sided linear bound.  A key difficulty is the lack of totally geodesic foliations in this more general context. To overcome this, we develop a comparison argument inspired by Choi~\cite{Choi} and G\'alvez--Lozano~\cite{GalvezLozano}, based on the construction of suitable convex barriers.

More precisely, our barriers are obtained through a ``scooping'' procedure that produces an unbounded convex set $\mathcal C\subset M^n$ with controlled geometry at infinity. This construction goes back to Anderson's work on the Dirichlet problem at infinity in pinched negatively curved manifolds~\cite{Anderson83} and was later refined in different directions, notably by Borb\'ely~\cite{Borbely} and, under weaker curvature assumptions, by Ripoll--Telichevesky~\cite{RipollTelichevesky}.

Anderson's construction yields a convex set $\mathcal C$ whose boundary is a priori only $C^0$. We pass to a parallel body $\mathcal C_\delta=\{d(\cdot,\mathcal C)\le\delta\}$, which is again convex and has boundary of class $C^{1,1}$. This $C^{1,1}$ regularity is all that our barrier argument needs, so no further smoothing of $\partial\mathcal C$ is required. Appendix~\ref{appendixA} gives the construction.

Our setting covers a wide class of manifolds, including all rank-one symmetric spaces of noncompact type: the real hyperbolic space $\mathbb{H}^n$, which corresponds to the case considered by Bonorino and Klaser, as well as the complex hyperbolic space $\mathbb{CH}^{m}$, of real dimension $2m$, the quaternionic hyperbolic space $\mathbb{HH}^{m}$, of real dimension $4m$, and the Cayley hyperbolic plane $\mathbb{OH}^2$, of real dimension $16$. These are all Cartan--Hadamard manifolds with pinched negative sectional curvature, and thus fall within the scope of our analysis.

Moreover, our method also covers bounded viscosity solutions of $\Delta u+f(u)=0$. The idea is not to compare $u$ with sub- or supersolutions of the semilinear equation, for which a comparison principle need not hold, but to reduce the problem to a linear eigenvalue estimate on a bounded set. The barriers are built from the convex sets above. The difference between $u$ and a finite sum of barriers satisfies $\Delta w+\lambda w\ge0$, in the viscosity sense, on a bounded set where it is positive, and the first eigenvalue of a ball containing this set gives the contradiction. We explain the mechanism after Corollary~\ref{cor:AC}. The viscosity framework and the subtraction argument are given in Section~\ref{subsection:viscosity}; see also~\cite{AzagraFerreraSanz2008}.

\subsection{Main results}

Throughout, $n\ge2$, $\overline\Omega$ denotes the closure in $M$; no continuity or boundary value at infinity is imposed on $u$. If $K\le-\kappa_2<0$, we use the visual metric
\[
d_o(\xi,\eta)=e^{-\sqrt{\kappa_2}(\xi|\eta)_o}
\]
on $\partial_\infty M^n$. Here $(\xi|\eta)_o$ is the Gromov product. We write $\mathcal H^s_{d_o}$ for Hausdorff measure and $\dim_H$ for Hausdorff dimension in this metric; the definitions are recalled in Section~\ref{subsect:HausdorffDimension}. Zero Hausdorff measure at a fixed exponent and Hausdorff dimension are independent of $o$, but the visual parameter $\sqrt{\kappa_2}$ is fixed.

\Needspace{6\baselineskip}
\begin{theorem}\label{thm:main}
Let $M^n$, $n\ge2$, be a Cartan--Hadamard manifold with pinched sectional curvature $-\kappa_1\le K\le-\kappa_2<0$. Fix
\[
 0\le\lambda\le\frac{(n-1)^2\kappa_2}{4},\quad
 s_\lambda=\frac{n-1}{2}
       +\sqrt{\frac{(n-1)^2}{4}-\frac{\lambda}{\kappa_2}}.
\]
Equip $\partial_\infty M$ with the visual metric $d_o(\xi,\eta)=e^{-\sqrt{\kappa_2}(\xi\mid\eta)_o}$.
Let $\Omega\subset M$ be a domain such that $\mathcal H^{s_\lambda}_{d_o}(\partial_\infty\Omega)=0$, and let
$f\in C(\mathbb R)$ satisfy $t f(t)\le\lambda t^2$ for all $t\in\mathbb R$.
Then the only bounded classical solution
$u\in C^2(\Omega)\cap C^0(\overline\Omega)$ of
\[
 \begin{cases}
 \Delta u+f(u)=0&\text{in }\Omega,\\
 u=0&\text{on }\partial\Omega
 \end{cases}
\]
is $u\equiv0$. The same conclusion holds for bounded viscosity solutions in $C^0(\overline\Omega)$ with these boundary values.
\end{theorem}

The vanishing of $\mathcal H^{s_\lambda}_{d_o}$ is independent of the
base point, since visual metrics with this fixed parameter are bi-Lipschitz
equivalent. The hypothesis is satisfied whenever
$\dim_H(\partial_\infty\Omega)<s_\lambda$.

In $\mathbb H^n(-1)$, taking $\kappa_1=\kappa_2=1$ gives
$s_\lambda=(n-1)/2+\sqrt{(n-1)^2/4-\lambda}$.
At $\lambda=(n-1)^2/4$, our criterion for $f(t)=\lambda t$ agrees exactly
with Bonorino--Klaser~\cite[Theorem~1.1]{BonorinoKlaser}, and their
Theorem~4.5 proves dimensional sharpness at this endpoint.
For each fixed $\lambda<(n-1)^2/4$, the exponent is larger than $(n-1)/2$.
This refinement of their uniform criterion already follows by retaining
the faster decay rate in their Lemma~3.1 throughout their covering
argument, instead of replacing it by $(n-1)/2$ in Lemma~3.5.
Their sharpness examples use the critical value of $\lambda$ and do not
establish optimality of $s_\lambda$ at a fixed subcritical value.

Several important nonlinearities satisfy the hypotheses of Theorem~\ref{thm:main}, including $f(t)=\tanh t$, $f(t)=\arctan t$, and $f(t)=\frac{t}{1+t^2}$, provided that $1\le(n-1)^2\kappa_2/4$: each satisfies $t f(t)\le t^2$, so one may take $\lambda=1$. The Allen--Cahn nonlinearity satisfies
$t(t-t^3)=t^2-t^4\le t^2$ on all of $\mathbb R$. Thus the following corollary uses $\lambda=1$ directly.

\begin{corollary}\label{cor:AC}
Let $n\ge3$ and let $M^n$ be a Cartan--Hadamard manifold with
$-\kappa\le K\le-1$, where $\kappa\ge1$. Set
\[
 s_{\mathrm{AC}}=\frac{n-1}{2}+\sqrt{\frac{(n-1)^2}{4}-1}.
\]
If a domain $\Omega\subset M$ satisfies
$\mathcal H^{s_{\mathrm{AC}}}_{d_o}(\partial_\infty\Omega)=0$
for $d_o(\xi,\eta)=e^{-(\xi\mid\eta)_o}$, then there is no nonzero bounded
classical or continuous viscosity solution, continuous up to $\partial\Omega$, of
\begin{equation}\label{eq:AllenCahn}
 \begin{cases}
 \Delta u+u-u^3=0&\text{in }\Omega,\\
 u=0&\text{on }\partial\Omega.
 \end{cases}
\end{equation}
\end{corollary}

For example, in real dimension $n=3$, Corollary~\ref{cor:AC} applies whenever
$\mathcal H^1_{d_o}(\partial_\infty\Omega)=0$, for every fixed $\kappa\ge1$.
This includes all domains with countable ideal boundary, and more generally
those with $\dim_H(\partial_\infty\Omega)<1$.

Let us explain where the exponent $s_\lambda$ comes from. Fix a base point $o$ and a visual ball of small radius $r$. Corollary~\ref{cor:visual-barriers} gives a convex set $\mathcal C$ whose ideal boundary contains this ball and whose distance to $o$ is at least $\kappa_2^{-1/2}\log(1/r)$, up to an additive constant. The construction moves the base point along the ray to the center of the visual ball before applying Theorem~\ref{thm:Anderson}. Let $h$ be the decaying profile of Bonorino and Klaser\rev{, rescaled to the model space $\mathbb H^n(-\kappa_2)$}. The function $v=h\circ d(\cdot,\mathcal C)$ satisfies $\Delta v+\lambda v\le0$ at distance larger than a fixed $t_0$ from $\mathcal C$. Since $h$ decays at the faster rate
\[
 \beta_+=\frac{(n-1)\sqrt{\kappa_2}}2+\sqrt{\frac{(n-1)^2\kappa_2}4-\lambda},
\]
the value of $v$ at $o$ is at most a constant times $e^{-\beta_+\kappa_2^{-1/2}\log(1/r)}=r^{s_\lambda}$.

Now let $u$ be a nonzero bounded solution, positive at $o$ after replacing $(u,f)$ by $(-u,-f(-\,\cdot))$ if necessary. If $\mathcal H^{s_\lambda}_{d_o}(\partial_\infty\Omega)=0$, compactness gives finitely many such balls covering $\partial_\infty\Omega$ with $\sum_ir_i^{s_\lambda}$ as small as we want. We normalize each barrier to be $2\sup u$ at distance $t_0$ from its convex set. This normalization does not depend on the cover, so the sum $v$ of the barriers satisfies $v(o)<u(o)$. The half-spaces contained in the convex sets make the region where all barriers are defined bounded, and $v>u$ on its boundary. Thus $w=u-v$ is positive on a domain $D$ compactly contained in that region and vanishes on $\partial D$. Since $f(t)\le\lambda t$ for $t>0$ and $\Delta v+\lambda v\le0$, we obtain $\Delta w+\lambda w\ge0$ in the viscosity sense. In fact, $w$ is only of class $C^{1,1}$ even when $u$ is classical, which is why viscosity solutions enter the argument. Comparing $w$ with the first eigenfunction of a geodesic ball $B\supset\overline D$ gives $\lambda\ge\lambda_1(B)>\lambda_0(M^n)$, with no regularity of $\partial D$. This contradicts $\lambda\le(n-1)^2\kappa_2/4\le\lambda_0(M^n)$, the last inequality being McKean's~\cite{McKean}.

Section~\ref{Preliminaries} fixes the geometric and viscosity conventions. In Section~\ref{Distance} we prove the distance comparison and construct supersolutions using the profile of Bonorino--Klaser. Section~\ref{Barriers} produces the convex barriers in the pinched setting. Section~\ref{Main} proves a general theorem under quantitative barrier assumptions and deduces Theorem~\ref{thm:main}. Appendix~\ref{appendixA} proves the convex construction.

\section{Preliminary Results}\label{Preliminaries}

In this section, we present some preliminary results concerning Cartan--Hadamard manifolds and the geometry and Hausdorff measure of their ideal boundaries. Our main references are \cite{BridsonHaefliger}, \cite{Choi}, \cite{EO}, \cite{GalvezLozano}, \cite{SY94}. We also recall viscosity solutions and the spectral comparison used below. Our basic references are \cite{AzagraFerreraSanz2008} for viscosity theory on manifolds. For background on the linear case $f(t)=\lambda t$, see the eigenvalue/maximum principle characterization \cite{BNV1994} and the viscosity eigenvalue theory for fully nonlinear operators \cite{QuaasSirakov2006,QuaasSirakov2008}.

\subsection{Asymptotic geometry of domains in Cartan--Hadamard manifolds}
\label{Asymptotic}

A Cartan--Hadamard manifold is a complete, simply connected Riemannian manifold
with nonpositive sectional curvature, i.e., $K \le 0$. Such manifolds are
diffeomorphic to $\mathbb R^n$ since, for any fixed base point $o\in M^n$, the
exponential map $\exp_o:T_oM^n\to M^n$ is a global diffeomorphism.

One of the main features of Cartan--Hadamard manifolds is the existence of the
\emph{ideal boundary} (or \emph{boundary at infinity}), denoted by
$\partial_\infty M^n$. This boundary consists of equivalence classes of
unit-speed geodesic rays, where two rays $\gamma_1,\gamma_2:[0,\infty)\to M^n$
are said to be equivalent if the distance between them remains bounded as
$t\to\infty$, namely if
\[
\limsup_{t\to\infty} \dist\big(\gamma_1(t),\gamma_2(t)\big)<\infty.
\]
It is shown in \cite{EO} that $\bar M^n:=M^n \cup \partial_\infty M^n$ can be endowed
with a natural topology, the \emph{cone topology}, which makes $\bar M^n$
homeomorphic to a closed topological ball.

Fix from now on a base point $o\in M^n$. Write
$S_oM^n:=\{v\in T_oM^n:|v|=1\}$ for its unit tangent sphere.
For distinct $x,y\in M^n$, let $\gamma_{x,y}$
denote the unique geodesic segment joining $x$ to $y$, parametrized by arc
length; it is unique because $\exp_x$ is a global diffeomorphism. We also use the
following notation for geodesic rays: given $z\in S_oM^n$, let
$\gamma_{o,z}:[0,\infty)\to M^n$ be the geodesic ray with initial conditions
$\gamma_{o,z}(0)=o$ and $\gamma_{o,z}'(0)=z$.

If $\gamma:[0,\infty)\to M^n$ is a geodesic ray and $[\gamma]=\xi\in\partial_\infty M^n$
is its equivalence class, we write
\(
\lim_{t\to\infty}\gamma(t)=\xi.
\)
The ideal boundary $\partial_\infty M^n$ can be identified with the unit tangent
sphere $S_oM^n \subset T_oM^n $ through the \emph{horizon map} $\textup{Hor}_o:S_oM^n\to\partial_\infty M^n$
defined by
\[
\textup{Hor}_o(z)=\lim_{t\to\infty}\gamma_{o,z}(t),\quad \forall\,z\in S_oM^n.
\]
This map is a homeomorphism (see, for instance, \cite{EO}). For $\xi\in\partial_\infty M$, we denote by $\gamma_{o,\xi}$ the ray with initial direction $\operatorname{Hor}_o^{-1}(\xi)$. In particular, once
$o\in M^n$ is fixed, we may view $\partial_\infty M^n$ as the space of directions $S_oM^n$.

Given $z\in S_oM^n$ and an angle
$\theta\in(0,\pi)$, we define the geodesic cone with vertex $o$, direction $z$,
and aperture $2\theta$ by
\[
C_o(z,\theta)
=\{x\in M^n\setminus\{o\}:\measuredangle_o(\gamma_{o,x}'(0),z)<\theta\},
\]
where $\measuredangle_o(\gamma_{o,x}'(0),z)$ denotes the angle at $o$ between
$\gamma_{o,x}'(0)$ and $z$. Cones with aperture $\pi$ will be called
\emph{half-spaces}; we write
\[
 H_o(z):=C_o(z,\pi/2).
\]
Its ideal boundary is the closed hemisphere
\[
 \partial_\infty H_o(z)
 =\operatorname{Hor}_o\{v\in S_oM:\measuredangle_o(v,z)\le\pi/2\},
\]
whose interior in $\partial_\infty M$ is the image of the open hemisphere
$\{v:\measuredangle_o(v,z)<\pi/2\}$. These descriptions follow from the
cone topology and the fact that $\operatorname{Hor}_o$ is a homeomorphism.
We write $B_p(r)$ and $S_p(r)=\partial B_p(r)$ 
for the open geodesic ball and sphere centered at any $p\in M$. For the fixed base point $o$, set
\(
\mathcal T_o(z,\theta,r):=C_o(z,\theta)\setminus \overline{B}_o(r).
\)
The open subsets of $M$, together with the sets
\[
\mathcal T_o(z,\theta,r)\ \cup\
\operatorname{Hor}_o\{v\in S_oM:\measuredangle_o(v,z)<\theta\},
\]
form a basis for the cone topology on $M\cup\partial_\infty M$. This topology does not depend on $o$ \cite[Proposition~II.8.8]{BridsonHaefliger}; in particular, $\operatorname{Hor}_p$ is a homeomorphism for every $p\in M$.

\subsection{Hausdorff dimension of a set in the asymptotic boundary}\label{subsect:HausdorffDimension}

Suppose that $(M^n,g)$ is a Cartan--Hadamard manifold with
$K\le-\kappa<0$. We fix the visual parameter $\sqrt\kappa$ throughout
this subsection. For $o,p,q\in M$, the \emph{Gromov product} at $o$ is
\begin{equation}\label{eq:Gromovproduct}
(p|q)_o:=\tfrac12\bigl(d(o,p)+d(o,q)-d(p,q)\bigr).
\end{equation}
For $\xi,\eta\in\partial_\infty M$, it extends to
\[
(\xi|\eta)_o:=
\sup_{p_i\to\xi,\ q_i\to\eta}\liminf_{i\to\infty}(p_i|q_i)_o;
\]
see \cite[Definition~III.H.3.15, p.~432]{BridsonHaefliger}.
This quantity is finite when $\xi\ne\eta$, whereas
$(\xi|\xi)_o=+\infty$
\cite[Remark~III.H.3.17(2), p.~432]{BridsonHaefliger}.
In our curvature setting, after rescaling the metric by $\sqrt\kappa$, one may equivalently take the limit along the
two geodesic rays from $o$ to $\xi$ and $\eta$; see
\cite[Exercise~III.H.3.18(2), p.~433]{BridsonHaefliger}.

The curvature bound and the Cartan--Hadamard theorem imply that $M$ is a $\mathrm{CAT}(-\kappa)$ space; see \cite[Theorems~II.1A.6 and~II.4.1]{BridsonHaefliger}. After rescaling the distance by $\sqrt\kappa$, Bourdon's construction
\cite[\S1.1, Proposition~(a), p.~98]{bourdon1996} gives the \emph{visual metric}
\begin{equation}\label{eq:visualMetric}
d_o(\xi,\eta) := e^{-\sqrt{\kappa}(\xi|\eta)_o},
\qquad \xi,\eta\in\partial_\infty M^n,
\end{equation}
where $e^{-\infty}=0$. Its topology is the cone topology: convergence
$\xi_j\to\xi$ in that topology is equivalent to
$(\xi_j|\xi)_o\to+\infty$; see
\cite[Proposition~III.H.3.7(2), p.~429, and Remark~III.H.3.17(6), p.~433]{BridsonHaefliger}.
We now record how a change of base point affects this metric and its
Hausdorff measures.

For a metric space $(X,d)$, a set $A\subset X$ and $s,\delta>0$, write
\[
\mathcal H^s_{d,\delta}(A):=
\inf\left\{\sum_j(\operatorname{diam}_d E_j)^s:
A\subset\bigcup_j E_j,\quad\operatorname{diam}_d E_j<\delta\right\},
\]
where the infimum is taken over countable covers by subsets of $X$. Set
\[
\mathcal H^s_d(A):=\lim_{\delta\downarrow0}\mathcal H^s_{d,\delta}(A),
\qquad
\dim_H(A,d):=\inf\{s>0:\mathcal H^s_d(A)=0\}.
\]

\begin{proposition}\label{prop:HD-indep-basepoint}\label{lem:visual-bilipschitz}
Let $o_1,o_2\in M$, $D=d(o_1,o_2)$ and $L=e^{\sqrt\kappa D}$. For all
$\xi,\eta\in\partial_\infty M$,
\begin{equation}\label{eq:Gromov-change-boundary}
(\xi|\eta)_{o_2}-D\le(\xi|\eta)_{o_1}\le(\xi|\eta)_{o_2}+D,
\end{equation}
with the inequalities understood in the extended real sense, and
\begin{equation}\label{eq:visual-bilipschitz}
L^{-1}d_{o_2}(\xi,\eta)\le d_{o_1}(\xi,\eta)\le Ld_{o_2}(\xi,\eta).
\end{equation}
Moreover, for every $A\subset\partial_\infty M$ and $s>0$,
\[
L^{-s}\mathcal H^s_{d_{o_1}}(A)
\le\mathcal H^s_{d_{o_2}}(A)
\le L^s\mathcal H^s_{d_{o_1}}(A).
\]
Thus zero Hausdorff measure at a fixed exponent and Hausdorff dimension
are independent of the base point.
\end{proposition}

\begin{proof}
For $p,q\in M$, the triangle inequality gives
\begin{equation}\label{eq:Gromov-change-base}
\bigl|(p|q)_{o_1}-(p|q)_{o_2}\bigr|
\le\tfrac12\bigl(|d(o_1,p)-d(o_2,p)|+|d(o_1,q)-d(o_2,q)|\bigr)
\le D.
\end{equation}
Taking the lower limit along any pair of sequences tending to $\xi,\eta$,
and then the supremum over such pairs, gives
\eqref{eq:Gromov-change-boundary}. When $\xi=\eta$, both products are
$+\infty$, so the additive inequalities remain valid.
Exponentiating gives \eqref{eq:visual-bilipschitz}; on the diagonal,
all three distances are zero.

Every cover admissible for $\mathcal H^s_{d_{o_1},\delta}(A)$ is admissible
for $\mathcal H^s_{d_{o_2},L\delta}(A)$, with sum at most $L^s$ times
as large. Hence
\[
\mathcal H^s_{d_{o_2},L\delta}(A)
\le L^s\mathcal H^s_{d_{o_1},\delta}(A).
\]
Letting $\delta\downarrow0$ and interchanging $o_1,o_2$ proves the measure
comparison. The assertions about zero measure and dimension follow.
\end{proof}

For a domain $\Omega\subset M^n$ we define its asymptotic boundary as
\[
\partial_\infty\Omega := \overline{\Omega} \cap\partial_\infty M^n,
\]
where the closure is taken in the cone topology on $M^n\cup\partial_\infty M^n$.
Given $o\in M^n$, we set
\begin{equation}\label{eq:HD-visual}
\dim_H(\partial_\infty\Omega)
:= \dim_H\bigl(\partial_\infty\Omega,d_o\bigr).
\end{equation}
Proposition~\ref{prop:HD-indep-basepoint} shows that this quantity is independent of $o$. We use $\mathcal H^s(\partial_\infty\Omega)$ for Hausdorff measure in the same visual metric. The set $\partial_\infty\Omega$ is compact, being closed in the compact ideal boundary.

We finish this section by relating the visual distance between two points in
$\partial_\infty M^n$ with the angle between the corresponding initial directions at $o$.

\begin{lemma}\label{lemmaDimension}
Let $(M^n,g)$ be Cartan--Hadamard with sectional curvature $K\le-\kappa<0$.
Consider $\xi,\eta\in\partial_\infty M^n$ and write
$\theta=\measuredangle_o(\gamma_{o,\xi}'(0),\gamma_{o,\eta}'(0))$. Then, defining $d_o$ as in
\eqref{eq:visualMetric},
\[
\sin(\theta/2)\le d_o(\xi,\eta).
\]
\end{lemma}

\begin{proof}
The case $\theta=0$ is immediate. Put $k=\sqrt\kappa$ and choose arbitrary
sequences $p_i\to\xi$, $q_i\to\eta$ in $M\setminus\{o\}$. Write
\[
a_i=d(o,p_i),\qquad b_i=d(o,q_i),\qquad
\alpha_i=\measuredangle_o(\gamma'_{o,p_i}(0),\gamma'_{o,q_i}(0)).
\]
Then $a_i,b_i\to\infty$ and $\alpha_i\to\theta$. Let $c_i$ be the third
side of the hinge $(a_i,b_i,\alpha_i)$ in $\mathbb H^2(-\kappa)$.
Hinge comparison \cite[Proposition~II.1.7(5) and Corollary~II.1A.7]{BridsonHaefliger}
gives $d(p_i,q_i)\ge c_i$, and the hyperbolic law of cosines gives
\[
\cosh(kc_i)=\sin^2(\alpha_i/2)\cosh(k(a_i+b_i))
+\cos^2(\alpha_i/2)\cosh(k(a_i-b_i)).
\]
Since $e^x\ge2\cosh x-1$ for $x\ge0$, we obtain
$e^{kc_i}\ge\sin^2(\alpha_i/2)e^{k(a_i+b_i)}-1$. For all sufficiently
large $i$, the right-hand side is positive, and therefore
\begin{equation}\label{eq:boundGromovProduct}
2k(p_i|q_i)_o\le k(a_i+b_i-c_i)
\le-\log\bigl(\sin^2(\alpha_i/2)-e^{-k(a_i+b_i)}\bigr).
\end{equation}
It follows that
$\limsup_i(p_i|q_i)_o\le-k^{-1}\log\sin(\theta/2)$.
Taking the supremum of the lower limits over all such pairs of sequences
and exponentiating yields $\sin(\theta/2)\le d_o(\xi,\eta)$.
\end{proof}

In the model case of the hyperbolic space $\mathbb{H}^n(-\kappa)$, the above lemma is sharp:
for any $\xi,\eta\in\partial_\infty\mathbb{H}^n(-\kappa)$ one has  $d_o(\xi,\eta)=\sin(\theta/2)$, $\theta=\measuredangle_o(\gamma_{o,\xi}'(0),\gamma_{o,\eta}'(0))$. 
In particular, the notion of Hausdorff dimension on the ideal boundary used in
\cite{BonorinoKlaser} agrees with the one in \eqref{eq:HD-visual}. Indeed, Bonorino and Klaser measure $\partial_\infty\mathbb H^n$ through the geodesic sphere of radius one centered at $o$. Since $\theta/\pi\le\sin(\theta/2)\le\theta/2$ for $\theta\in[0,\pi]$, its intrinsic and chordal metrics are bi-Lipschitz equivalent to $d_o$. Hence both notions have the same null sets and the same Hausdorff dimension.

\subsection{Viscosity solutions on Riemannian manifolds}\label{subsection:viscosity}

In this subsection we recall the notion of viscosity solution on Riemannian
manifolds and a spectral comparison for positive viscosity subsolutions, in the spirit of Berestycki--Nirenberg--Varadhan \cite{BNV1994}.

\begin{definition}\label{def:viscosity}
Let $(M^n,g)$ be a smooth Riemannian manifold and $\Omega\subset M^n$ an open set. Let $f\in C^0(\mathbb{R})$. A function $u\in C^0(\Omega)$ is a \emph{viscosity subsolution} of $\Delta u + f(u)=0$ in $\Omega$ if for every $p_0\in\Omega$ and every $\varphi\in C^2$ defined near $p_0$ such that $u-\varphi$ attains a local maximum at $p_0$, one has
\[
\Delta \varphi(p_0) + f\big(u(p_0)\big)\ge 0.
\]
A \emph{viscosity supersolution} is defined analogously, with local minimum and ``$\le 0$''. A \emph{viscosity solution} is both a sub- and a supersolution.

For Dirichlet data $u=\psi$ on $\partial\Omega$, we use the pointwise interpretation on the boundary: $u$ is a subsolution (resp.\ supersolution) in $\Omega$ and satisfies $\limsup_{p\to p_0}u(p)\le \psi(p_0)$ (resp.\ $\liminf_{p\to p_0}u(p)\ge \psi(p_0)$) for all $p_0\in\partial\Omega$.
\end{definition}

We use the following Barta-type comparison. Its spectral background is the principal-eigenvalue characterization of the maximum principle in \cite[Theorems~1.1--1.2 and Corollary~1.1]{BNV1994}. See also \cite[Theorem~1.5]{QuaasSirakov2008} for the corresponding spectral maximum principle in the $L^n$-viscosity framework. For continuous viscosity subsolutions we include a direct proof. It touches the subsolution from above by a multiple of the first eigenfunction, and uses neither a comparison principle nor any regularity of $\partial D$.

\begin{proposition}[Spectral comparison on an enclosing domain]\label{prop:Lambda}
Let $D\subset M^n$ be a bounded domain in a Cartan--Hadamard manifold, and let $v\in C^0(\overline D)$ satisfy $v>0$ in $D$, $v=0$ on $\partial D$, and
\[
\Delta v+\lambda v\ge0\quad\hbox{in }D
\]
in the viscosity sense. If $G$ is a smooth, connected, relatively compact domain with $\overline D\subset G$, then $\lambda\ge\lambda_1(G)$, where $\lambda_1(G)$ is the first Dirichlet eigenvalue of $-\Delta$ on $G$.
\end{proposition}

\begin{proof}
Let $\varphi_1$ be a smooth first Dirichlet eigenfunction of $G$, positive in $G$, so that $\Delta\varphi_1+\lambda_1(G)\varphi_1=0$. Since $\overline D$ is a compact subset of $G$, $\varphi_1$ has a positive minimum on $\overline D$\rev{. Hence the quotient $v/\varphi_1$ is continuous on $\overline D$, and $t_*:=\max_{\overline D}(v/\varphi_1)$ is positive, since $v>0$ in $D$.} The function $v-t_*\varphi_1$ is nonpositive on $\overline D$, vanishes somewhere, and is negative on $\partial D$, where $v=0$. Hence it attains its maximum value $0$ at some $x_0\in D$. Testing Definition~\ref{def:viscosity} with $t_*\varphi_1$ at $x_0$ gives
\[
0\le\Delta(t_*\varphi_1)(x_0)+\lambda v(x_0)=(\lambda-\lambda_1(G))\,v(x_0).
\]
Since $v(x_0)=t_*\varphi_1(x_0)>0$, we obtain $\lambda\ge\lambda_1(G)$.
\end{proof}

The barriers constructed in the sequel are of class $C^{1,1}$ rather than
smooth, so we shall subtract a $C^{1,1}$ function from a viscosity subsolution
and need to know that the difference is again a viscosity subsolution of the
resulting linear equation. 

Since our setting
involves only the Laplace--Beltrami operator and a $C^{1,1}$ barrier, we prefer
to give an elementary self-contained proof below, based on mollification and
the stability of viscosity subsolutions under uniform limits.

\begin{lemma}[Difference of a viscosity subsolution and a $C^{1,1}$ barrier]\label{lem:difference}
Let $M^n$ be a Cartan--Hadamard manifold and let $O\subset M^n$ be an open set. Assume:
\begin{itemize}
\item[(i)] $u\in C^0(O)$ is a viscosity subsolution of $\Delta u+f(u)=0$ in $O$, where
$f\in C^0(\mathbb{R})$ satisfies $f(s)\le\lambda s$ for all $s>0$ and some $\lambda\ge0$;
\item[(ii)] $v\in C^{1,1}_{\mathrm{loc}}(O)$ satisfies $\Delta v+\lambda v\le 0$ at almost every
point of $O$ (in the sense of the pointwise second-order derivatives provided by Rademacher's theorem).
\end{itemize}
Then $w:=u-v$ is a viscosity subsolution of
\[
\Delta w+\lambda w=0 \text{ in }O^+:=O\cap\{u>0\}.
\]
\end{lemma}

\begin{proof}
Since the notion of viscosity subsolution is local, it suffices to fix $q\in O^+$ and to verify
the subsolution property on a neighborhood of $q$. Choose $r>0$ so small that the geodesic ball
$B_{2r}:=B_{2r}(q)$ is contained in $O^+$, and identify $B_{2r}$ with the Euclidean ball
$\{|x|<2r\}\subset\mathbb{R}^n$ through normal coordinates centered at $q$; recall that $\exp_q$
is a global diffeomorphism and maps Euclidean balls onto geodesic balls. In these coordinates the
Laplace--Beltrami operator reads
\[
\Delta\phi=\sum_{i,j=1}^n a^{ij}\,\partial^2_{ij}\phi
             +\sum_{i=1}^n b^i\,\partial_i\phi,
\qquad a^{ij}:=g^{ij},\quad
b^i:=-\sum_{j,k=1}^n g^{jk}\Gamma^i_{jk},
\]
with $a^{ij},b^i\in C^\infty$ and $(a^{ij})$ uniformly elliptic on $\overline B_{3r/2}$.

We first mollify $v$ with a controlled error. Since $v\in C^{1,1}(\overline B_{3r/2})\subset W^{2,\infty}(B_{3r/2})$, 
its first and second weak coordinate derivatives are bounded. Set
\[
 A:=\max_{i,j}\|\partial_{ij}v\|_{L^\infty(B_{3r/2})},\qquad
 B:=\max_i\|\partial_i v\|_{L^\infty(B_{3r/2})}.
\]
Let $\omega_a,\omega_b$ be the componentwise moduli of continuity of $(a^{ij})$ and $(b^i)$ on $\overline B_{3r/2}$. Fix $\zeta\in C^\infty_c(B_1)$ with $\zeta\ge0$ and $\int\zeta=1$, and, for $0<\tau<r/2$\rev{,} put
$\zeta_\tau(z):=\tau^{-n}\zeta(z/\tau)$ and $v^\tau:=v*\zeta_\tau\in C^\infty(B_r)$.
Then $\|v^\tau-v\|_{C^0(B_r)}\le\sqrt n B\tau$.
Weak differentiation commutes with convolution:
\[
\partial_{ij}v^\tau=(\partial_{ij}v)*\zeta_\tau,
\qquad \partial_iv^\tau=(\partial_iv)*\zeta_\tau;
\]
see \cite[\S5.3.1, proof of Theorem~1, formula~(1)]{Evans1998}. For the Lipschitz functions $\partial_iv$, the derivatives given almost everywhere by Rademacher's theorem coincide with their weak derivatives, so hypothesis~(ii) holds with weak derivatives.
The coefficients of $\Delta$, however, depend on the point. Adding and subtracting their values at $x-z$ gives, for $x\in B_r$,
\[
\Delta v^\tau(x)+\lambda v^\tau(x)
=\int_{B_\tau}[\Delta v+\lambda v](x-z)\zeta_\tau(z)\,dz+E_\tau(x),
\]
where
\begin{align*}
E_\tau(x)
={}&\sum_{i,j}\int_{B_\tau}\bigl[a^{ij}(x)-a^{ij}(x-z)\bigr]
       \partial_{ij}v(x-z)\zeta_\tau(z)\,dz\\
&+\sum_i\int_{B_\tau}\bigl[b^i(x)-b^i(x-z)\bigr]
       \partial_i v(x-z)\zeta_\tau(z)\,dz.
\end{align*}
Since $\zeta_\tau$ has unit integral and support in $B_\tau$,
\[
|E_\tau(x)|\le n^2A\,\omega_a(\tau)+nB\,\omega_b(\tau)=:\omega(\tau)
\longrightarrow0
\]
uniformly for $x\in B_r$. The constant term $\lambda$ commutes with convolution and contributes no error. The convolution of $\Delta v+\lambda v$ is nonpositive by hypothesis (ii) and $\zeta_\tau\ge0$; no sign is asserted for the commutator $E_\tau$.

Thus
\begin{equation}\label{eq:molli}
\Delta v^\tau+\lambda v^\tau\le\omega(\tau)\quad\text{pointwise on }B_r,
\qquad\omega(\tau)\xrightarrow[\tau\to0]{}0.
\end{equation}

For fixed $\tau$, the continuous difference $w^\tau:=u-v^\tau$ is a viscosity subsolution of
$\Delta w+\lambda w+\omega(\tau)=0$ on $B_r$. Indeed, if $\varphi\in C^2$ is such that
$w^\tau-\varphi$ has a local maximum at $p_1\in B_r$, then $u-(\varphi+v^\tau)$ has a local maximum
at $p_1$ with $\varphi+v^\tau\in C^2$, so Definition~\ref{def:viscosity} applied to $u$ gives
$\Delta(\varphi+v^\tau)(p_1)+f(u(p_1))\ge0$. As $u(p_1)>0$, the one-sided bound on $f$ yields
$f(u(p_1))\le\lambda u(p_1)$, and \eqref{eq:molli} gives
\[
\begin{aligned}
\Delta\varphi(p_1)&\ge-\Delta v^\tau(p_1)-f(u(p_1))\\
&\ge\lambda v^\tau(p_1)-\omega(\tau)-\lambda u(p_1)
=-\lambda w^\tau(p_1)-\omega(\tau).
\end{aligned}
\]

It remains to pass to the limit $\tau\to0$. We spell out the usual stability argument for locally uniform limits; see \cite[\S6]{CrandallIshiiLions1992}. Let $\varphi\in C^2$ be such that $w-\varphi$ has a
local maximum at $p_0\in B_r$. Replacing $\varphi$ by $\varphi_\epsilon:=\varphi+\epsilon\,
\rho_{p_0}^2$, where $\rho_{p_0}:=d(p_0,\cdot)$ and $\epsilon>0$, turns this into a \emph{strict}
local maximum; note $\rho_{p_0}^2\in C^\infty$ because $M^n$ is Cartan--Hadamard. Choose $s>0$ with
$\overline B_s(p_0)\subset B_r$ such that $p_0$ is the unique maximum of $w-\varphi_\epsilon$ on
$\overline B_s(p_0)$, and let $p_\tau$ maximize $w^\tau-\varphi_\epsilon$ over $\overline B_s(p_0)$.
Since $w^\tau\to w$ uniformly on $\overline B_s(p_0)$, every accumulation point of $(p_\tau)$
maximizes $w-\varphi_\epsilon$, so $p_\tau\to p_0$ and, for $\tau$ small, $p_\tau$ is an interior
maximum of $w^\tau-\varphi_\epsilon$. Applying the preceding paragraph at $p_\tau$,
\[
\Delta\varphi_\epsilon(p_\tau)+\lambda w^\tau(p_\tau)+\omega(\tau)\ge0,
\]
and letting $\tau\to0$, using $\varphi_\epsilon\in C^2$, $p_\tau\to p_0$ and $w^\tau\to w$
uniformly, yields $\Delta\varphi(p_0)+\epsilon\,\Delta\rho_{p_0}^2(p_0)+\lambda w(p_0)\ge0$.
Since $\Delta\rho_{p_0}^2(p_0)=2n$, this reads $\Delta\varphi(p_0)+2n\epsilon+\lambda w(p_0)\ge0$,
and letting $\epsilon\downarrow0$ gives $\Delta\varphi(p_0)+\lambda w(p_0)\ge0$. As $q\in O^+$,
$\varphi$ and $p_0$ were arbitrary, $w$ is a viscosity subsolution of $\Delta w+\lambda w=0$ in
$O^+$.
\end{proof}

\section{Construction of super-solutions}\label{Distance}

Throughout this section, $M^n$, $n\ge2$, is a Cartan--Hadamard manifold with
$K\le-\kappa<0$. We construct supersolutions outside closed convex sets. This
formulation includes the parallel sets used in Section~\ref{Barriers} and
allows their boundaries to have only $C^{1,1}$ regularity.

The construction has two ingredients. We first obtain a lower bound for the
Laplacian of the distance from a convex set. We then choose a positive decreasing
solution of the corresponding one-dimensional equation. The sign of its
derivative turns the lower bound for the distance Laplacian into the upper bound
needed for a supersolution.

\subsection{The distance from a convex set}

For a nonempty closed convex set $C\subset M$, let us write
\[
 d_C(p):=\inf_{q\in C}d(p,q).
\]
The nearest point to $p$ in $C$ is denoted by $P_Cp$. If $p\notin C$, then
$P_Cp\in\partial C$ and the unique minimizing segment from $P_Cp$ to $p$ has
length $d_C(p)$. Thus, outside $C$, this is the distance to its boundary, with
the positive sign corresponding to the outward direction. The projection $P_C$ 
is single-valued and nonexpansive; see~\cite[Proposition~II.2.4(1),(4)]{BridsonHaefliger}. 
We recall the following standard regularity result from
\cite[Lemma~2.3(2) and equation~(2.2)]{BallmannPolymerakis2020}.

\begin{lemma}\label{lem:convex-distance-regularity}
Let $M$ be any Cartan--Hadamard manifold and let $C\subset M$ be a nonempty proper closed convex set. 
Its distance function $d_C$ belongs to $C^{1,1}_{\mathrm{loc}}(M\setminus C)$ and satisfies
$|\nabla d_C|=1$ there. No global lower curvature bound is needed for this
local regularity statement.
\end{lemma}

Our curvature convention is $R(X,Y)Z=\nabla_X\nabla_YZ-\nabla_Y\nabla_XZ-\nabla_{[X,Y]}Z$. Thus Jacobi fields satisfy $J''+R(J,\dot\gamma)\dot\gamma=0$, and $\langle R(J,\dot\gamma)\dot\gamma,J\rangle=K(J,\dot\gamma)|J|^2$ when $J\perp\dot\gamma$ and $|\dot\gamma|=1$. We use the convention
$\II_\Sigma(X,Y)=-\langle\nabla_XN,Y\rangle$ for a hypersurface with unit
normal $N$. Thus, if $\Sigma=\partial C$ is $C^{1,1}$ and $N$ points out of
$C$, convexity gives $S_0:=\nabla N\ge0$, or equivalently
$\II_\Sigma\le0$, at almost every point of $\Sigma$.

Let $h$ be smooth and set $v=h\circ d_C$. At every point where $d_C$ is twice
differentiable, the chain rule gives $\nabla v=h'(d_C)\nabla d_C$, and, for tangent vectors $X,Y$ at that point,
\[
 \nabla^2v(X,Y)=h''(d_C)\langle\nabla d_C,X\rangle
 \langle\nabla d_C,Y\rangle+h'(d_C)\nabla^2d_C(X,Y).
\]
Since $|\nabla d_C|=1$, taking the trace yields
\(
 \Delta v=h''(d_C)+h'(d_C)\Delta d_C.
\)
We estimate the distance Hessian using normal Jacobi fields,
following the approach of G\'alvez--Lozano~\cite{GalvezLozano} and Choi~\cite{Choi}.
For the comparison with a totally geodesic hyperplane, the required
index-form bound has the direct proof below. The normal-flow calculation
for $C^{1,1}$ boundaries is justified in Lemma~\ref{EDO-v}.

\begin{lemma}\label{LemComparison}
Let $\gamma:[0,L]\to M$ be a unit-speed geodesic. For any $C^1$ vector field
$J\perp\dot\gamma$ along $\gamma$, define
\[
 \mathcal I_L(J,J)=\int_0^L
 \bigl(|J'|^2-\langle R(J,\dot\gamma)\dot\gamma,J\rangle\bigr)\,dt.
\]
Then,
\[
 \mathcal I_L(J,J)\ge\sqrt\kappa\tanh(\sqrt\kappa L)|J(L)|^2.
\]

In particular, the sum of these index forms for $n-1$ fields whose endpoint
values form an orthonormal basis of $\dot\gamma(L)^\perp$ is at least
$(n-1)\sqrt\kappa\tanh(\sqrt\kappa L)$.
\end{lemma}

\begin{proof}
Set $q(t)=\sqrt\kappa\tanh(\sqrt\kappa t)$, the logarithmic derivative of $\cosh(\sqrt\kappa t)$. It satisfies $q(0)=0$ and $q'+q^2=\kappa$. Expanding
\[
|J'-qJ|^2=|J'|^2+q^2|J|^2-q\frac{d}{dt}|J|^2
\]
and integrating the last term by parts gives
\[
 \int_0^L\bigl(|J'|^2+\kappa|J|^2\bigr)\,dt
 =q(L)|J(L)|^2+\int_0^L|J'-qJ|^2\,dt.
\]
The sectional curvature bound gives
$-\langle R(J,\dot\gamma)\dot\gamma,J\rangle\ge\kappa|J|^2$
pointwise, including points where $J=0$. The result follows.
\end{proof}

\subsection{A decaying one-dimensional profile}

For a function depending only on signed distance from a totally geodesic
hyperplane in $\mathbb H^n(-\kappa)$, the equation
$\Delta v+\lambda v=0$ reduces to
\begin{equation}\label{EDO_Hn}
 h''(t)+(n-1)\sqrt\kappa\tanh(\sqrt\kappa t)h'(t)+\lambda h(t)=0.
\end{equation}

The following result is Bonorino--Klaser~\cite[Lemma~3.1]{BonorinoKlaser},
rescaled to curvature $-\kappa$. If $H$ solves their equation with spectral
parameter $\lambda/\kappa$, then $h(t)=H(\sqrt\kappa t)$ solves~\eqref{EDO_Hn}.

As $t\to\infty$, the coefficient of $h'$ tends to $(n-1)\sqrt\kappa$.
The two exponential rates of the limiting equation are
\[
 \frac{(n-1)\sqrt\kappa}{2}
 \mathbin{\pm}\sqrt{\frac{(n-1)^2\kappa}{4}-\lambda}.
\]
We retain the larger rate, which gives the exponent in the Hausdorff
measure condition of Theorem~\ref{thm:main}. The estimate below also holds
at the endpoint where the two rates coincide.

\begin{lemma}\label{lem:fast-profile}
Set $\lambda_*:=\frac{(n-1)^2\kappa}{4}$ and 
$ \beta_+(\lambda):=\frac{(n-1)\sqrt\kappa}{2}+\sqrt{\lambda_*-\lambda}$. For every $0\le\lambda\le\lambda_*$, equation~\eqref{EDO_Hn} has a smooth
solution $h$ and constants $t_0>0$, $C_h>0$ such that
\[
 h(t)>0,\qquad h'(t)<0,\qquad
 h(t)\le C_he^{-\beta_+(\lambda)t},\quad t\ge t_0.
\]
\end{lemma}

\subsection{Supersolutions with \texorpdfstring{$C^{1,1}$}{C1,1} regularity}

\begin{lemma}\label{EDO-v}
Let $C\subset M$ be a proper closed convex set with nonempty interior and
boundary $\Sigma$ of class $C^{1,1}$. Set $d=d_C$ on $M\setminus C$.
If $h$ and $t_0$ are given by Lemma~\ref{lem:fast-profile}, then
$v=h\circ d$ belongs to $C^{1,1}_{\mathrm{loc}}(\{d>t_0\})$ and satisfies
\(
 \Delta v+\lambda v\le0
\)
almost everywhere and in the viscosity sense on $\{d>t_0\}$.
At points where the distance is smooth the inequality is classical.
\end{lemma}

\begin{proof}
Lemma~\ref{lem:convex-distance-regularity} gives the asserted regularity.
We first justify the normal Jacobi-field calculation at almost every
exterior point. Let $N$ be the outward unit normal to $\Sigma$. The map
\[
 F:\Sigma\times(0,\infty)\longrightarrow M\setminus C,
 \qquad F(y,t)=\exp_y(tN(y)),
\]
is a bijection with inverse $x\mapsto(P_Cx,d_C(x))$. Indeed, first variation
shows that the segment from $P_Cx$ to $x$ is normal to $\Sigma$. Conversely,
convexity and the supporting-normal inequality show that $y$ is the nearest
point of $C$ to $\exp_y(tN(y))$ for every $t>0$. Both $F$ and its inverse
are locally Lipschitz in charts: this follows from $N\in C^{0,1}_{\mathrm{loc}}$,
smoothness of the exponential map, and nonexpansiveness of $P_C$.

Let $Z\subset\Sigma$ be the set, of $(n-1)$-dimensional measure zero, where
$N$ is not differentiable or the second fundamental form is not defined.
At a point $y\notin Z$, $S_0X:=\nabla_XN$ is self-adjoint, by symmetry of second derivatives, and nonnegative,
by convexity. On a countable collection of relatively compact boundary
charts and bounded time intervals, Fubini's theorem and Lipschitz
preservation of null sets give $\mathcal H^n\bigl(F(Z\times(0,\infty))\bigr)=0$.

For such a good footpoint $y$ and $X\in T_y\Sigma$, differentiation of the
normal flow gives a perpendicular Jacobi field along
$\gamma(t)=F(y,t)$, with $J_X(0)=X$ and  $J_X'(0)=S_0X$. Its squared norm satisfies
\begin{equation*}
\begin{split}
 (|J_X|^2)''&=2|J_X'|^2
   -2\langle R(J_X,\dot\gamma)\dot\gamma,J_X\rangle\ge0 ,
    \text{ and }\\
    (|J_X|^2)'(0)& =2\langle S_0X,X\rangle\ge0.
\end{split}
\end{equation*}

Consequently $|J_X(t)|\ge|X|$. The tangential differential
$X\mapsto J_X(t)$ is invertible, and the full differential of $F$ is
invertible because its remaining direction is $\partial_tF=\dot\gamma$.
A locally bi-Lipschitz map differentiable at a point with invertible
differential has differentiable inverse there, as follows by substituting
its first-order expansion and using the Lipschitz bound for the inverse.
Thus the identity $\nabla d(F(y,t))=\dot\gamma_y(t)$ can be differentiated 
at every point arising from a good footpoint.

Fix $L>0$ and an orthonormal basis $e_1,\ldots,e_{n-1}$ of
$\dot\gamma(L)^\perp$. Choose $X_i$ and the corresponding fields $J_i$ so that
\begin{equation}\label{Jacobi}
 J_i(L)=e_i,\quad J_i(0)=X_i,\quad J_i'(0)=S_0X_i.
\end{equation}

The fields $J_i$ come from variations through normal geodesics. Therefore,
$\nabla_{J_i}\nabla d=J_i'$ along the variation, and at $t=L$ we obtain
\(
 \nabla^2d(e_i,e_i)=\langle J_i'(L),J_i(L)\rangle.
\)

On the other hand, the Jacobi equation gives
\[
 \frac{d}{dt}\langle J_i',J_i\rangle
 =|J_i'|^2-\langle R(J_i,\dot\gamma)\dot\gamma,J_i\rangle.
\]
Integrating from $0$ to $L$ and using $J_i'(0)=S_0X_i$, we obtain
\begin{equation*}
 \begin{split}
 \nabla^2d(e_i,e_i) &=\langle J_i'(L),J_i(L)\rangle=\langle S_0X_i,X_i\rangle+\mathcal I_L(J_i,J_i)\\
 &\ge\sqrt\kappa\tanh(\sqrt\kappa L),
 \end{split}
 \end{equation*}where Lemma~\ref{LemComparison} was used in the last line.
Also $\nabla^2d(\nabla d,\cdot)=0$, since $|\nabla d|=1$. Hence
\begin{equation}\label{eq:distance-laplacian}
 \Delta d\ge(n-1)\sqrt\kappa\tanh(\sqrt\kappa d)
 \quad\text{almost everywhere on }M\setminus C.
\end{equation}

This comparison is also proved by Ballmann and Polymerakis
\cite[Lemma~2.12(1)]{BallmannPolymerakis2020} for distance from an
arbitrary closed convex set. The argument above gives a self-contained
proof for the $C^{1,1}$ boundaries used here.

As $h'<0$ on $[t_0,\infty)$, the chain rule yields
\begin{equation}\label{Laplacian}
 \begin{aligned}
 \Delta v+\lambda v
 &=h''(d)+h'(d)\Delta d+\lambda h(d)\\
 &\le h''(d)+(n-1)\sqrt\kappa\tanh(\sqrt\kappa d)h'(d)+\lambda h(d)=0
 \end{aligned}
\end{equation}
almost everywhere on $\{d>t_0\}$. 

For completeness, this inequality also holds in the viscosity sense.
In a relatively compact coordinate chart, write
$\Delta=\sum_{i,j=1}^n a^{ij}(x)\partial_{ij}+\sum_{i=1}^n b^i(x)\partial_i$ and mollify $v$ to
$v_\tau=v*\zeta_\tau$ on a smaller chart. The coefficient calculation 
in Lemma~\ref{lem:difference} applies because $v\in W^{2,\infty}_{\mathrm{loc}}$ 
and the coefficients are smooth. It gives
\[
 \Delta v_\tau+\lambda v_\tau\le\omega(\tau),
 \qquad \omega(\tau)\longrightarrow0,
\]
uniformly on compact subcharts. Indeed, for the second-order terms the
error is bounded by the modulus of continuity of $a^{ij}$ times
$\|D^2v\|_{L^\infty}$, and the first-order terms are handled in the same
way. Moreover $v_\tau\to v$ locally uniformly.
If $\varphi\in C^2$ touches $v$ from below at $p$, replace it by
$\varphi_\varepsilon=\varphi-\varepsilon|x-x(p)|^2$ in the chart. On a
sufficiently small closed coordinate ball the contact is strict, so
minimizers $p_\tau$ of $v_\tau-\varphi_\varepsilon$ are interior and tend
to $p$. At these points,
\[
 \Delta\varphi_\varepsilon(p_\tau)+\lambda v_\tau(p_\tau)
 \le\Delta v_\tau(p_\tau)+\lambda v_\tau(p_\tau)
 \le\omega(\tau).
\]
Taking first $\tau\downarrow0$ and then $\varepsilon\downarrow0$ proves
$\Delta\varphi(p)+\lambda v(p)\le0$, as required.
\end{proof}

\section{Convex barriers in pinched Cartan--Hadamard manifolds}\label{Barriers}

In this section we show that Cartan--Hadamard manifolds with pinched negative
sectional curvature admit the convex hypersurfaces (barriers) needed in the proof
of our main theorem.

Throughout, $(M^n,g)$ denotes a Cartan--Hadamard manifold. When needed, we will impose
additional curvature assumptions, which will be stated explicitly.

Following the convex-set construction of Anderson~\cite{Anderson83},
we prove the next result in Appendix~\ref{appendixA}.

\begin{theorem}\label{thm:Anderson}
Let $M^n$ be Cartan--Hadamard with sectional curvature pinched between
$-\kappa_1\le K\le -\kappa_2<0$. Given $p\in M^n$ and $w\in S_pM^n$, there exists a
closed convex set $\mathcal{C}\subset M^n$ with boundary of class $C^{1,1}$, whose second
fundamental form satisfies $\II_{\partial\mathcal C}\le0$ a.e.\ with respect to the outward
unit normal, such that:
\begin{itemize}
\item[(i)] $H_p(-w)\subset \mathcal{C}$;
\item[(ii)] There exists $r_0>0$, depending only on $\kappa_1$ and $\kappa_2$, such that if
$t\ge r_0$, then
\(
d(\gamma_{p,w}(t),\mathcal{C})\ge t-r_0,
\)
where $\gamma_{p,w}(t):=\exp_p(tw)$.
\end{itemize}
\end{theorem}

The following consequence gives the visual-ball barriers used in
Section~\ref{Main}; see Figure~\ref{fig:Picture4}.

\begin{corollary}[Barriers for visual balls]\label{cor:visual-barriers}
Let $M^n$ be Cartan--Hadamard with $-\kappa_1\le K\le-\kappa_2<0$, and put
$k=\sqrt{\kappa_2}$. There exist $r_{\mathrm v}>0$ and $C_0>0$, depending only
on the curvature bounds, such that for every $o\in M$, every
$\xi\in\partial_\infty M$ and $0<r<r_{\mathrm v}$ there are a proper closed
convex set $\mathcal C$ with nonempty interior and $C^{1,1}$ boundary,
a point $q\in M$ and $\nu\in S_qM$ satisfying $H_q(\nu)\subset\mathcal C$, 
$B_{d_o}(\xi,r)\subset\operatorname{int}_{\partial_\infty M}\partial_\infty H_q(\nu)$, and
\[
d(o,\mathcal C)\ge\frac1k\log\frac1r-C_0>0.
\]
Here $d_o(\xi,\eta)=e^{-k(\xi\mid\eta)_o}$.
\end{corollary}

\begin{proof}
Let $r_0$ be the uniform constant in Theorem~\ref{thm:Anderson}. Set
\[
 T=\frac1k\log\frac1{2r},\qquad
 q=\gamma_{o,\xi}(T),\qquad \nu=\dot\gamma_{o,\xi}(T).
\]
Choose $r_{\mathrm v}<\frac12e^{-k(r_0+1)}$, so $T>r_0+1$.
By Proposition~\ref{prop:HD-indep-basepoint} and Lemma~\ref{lemmaDimension},
for every $\eta\in B_{d_o}(\xi,r)$,
\[
 \sin\!\left(\frac12\measuredangle_q
       (\nu,\operatorname{Hor}_q^{-1}(\eta))\right)
 \le d_q(\xi,\eta)\le e^{kT}d_o(\xi,\eta)<\frac12.
\]
Thus this ball lies in the open angular cap of radius $\pi/3$ at $q$,
strictly inside the ideal hemisphere of $H_q(\nu)$.
Apply Theorem~\ref{thm:Anderson} at $(q,-\nu)$. Its convex set contains
$H_q(\nu)$, and its separation ray reaches $o$ at time $T$, giving
\[
 d(o,\mathcal C)\ge T-r_0
 =\frac1k\log\frac1r-\left(r_0+\frac{\log2}{k}\right)>0.
\]
Take $C_0=r_0+k^{-1}\log2$. The set is proper by the distance estimate
and has nonempty interior because it contains the solid half-space.
\end{proof}

\begin{figure}[htbp]
  \centering
  \includegraphics[width=0.46\linewidth]{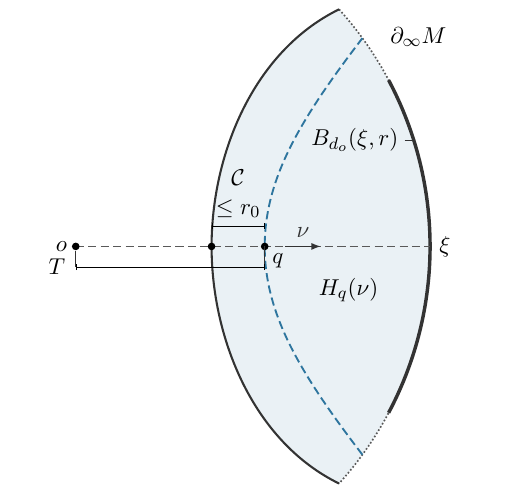}
  \caption{\small Schematic of the visual-ball barrier. The shaded convex
  set $\mathcal C$ contains $H_q(\nu)$. The heavy ideal arc represents
  $B_{d_o}(\xi,r)$, lying strictly inside the ideal hemisphere of
  $H_q(\nu)$. The point $q=\gamma_{o,\xi}(T)$ is at distance
  $T=k^{-1}\log(1/(2r))$ from $o$. The axial segment from
  $\partial\mathcal C$ to $q$ has length at most $r_0$; the separation
  estimate gives $d(o,\mathcal C)\ge T-r_0$.}
  \label{fig:Picture4}
\end{figure}

\FloatBarrier
\section{Non-existence results via convex barriers}\label{Main}

Let $M^n$, $n\ge2$, be a Cartan--Hadamard manifold with sectional curvature
$K\le-\kappa<0$. We state the argument in terms of convex barriers. In the
pinched setting these are supplied by Corollary~\ref{cor:visual-barriers}.
All Hausdorff measures in this section refer to the visual metric 
$d_o(\xi,\eta)=e^{-\sqrt\kappa(\xi\mid\eta)_o}$. The property of having zero Hausdorff measure of a fixed exponent is
independent of $o$, by the bi-Lipschitz comparison of these metrics.

We assume that there are constants $a>0$, $r_{\mathrm v}>0$ and
$C_1,C_2>0$ with the following property. For every $o\in M$,
$\xi\in\partial_\infty M$ and $0<r<r_{\mathrm v}$, there is a proper closed
convex set $\mathcal C\subset M$ with nonempty interior and $C^{1,1}$
boundary such that
\begin{itemize}
\item[(B1)] $o\notin\mathcal C$;
\item[(B2)] there are $q\in M$ and $\nu\in S_qM$ for which $H_q(\nu)\subset\mathcal C$ 
and $B_{d_o}(\xi,r)\subset \operatorname{int}_{\partial_\infty M}\partial_\infty H_q(\nu)$;
\item[(B3)] $d(o,\mathcal C)\ge a^{-1/2}\log(C_1/r)-C_2$.
\end{itemize}

The distance regularity required in Lemma~\ref{EDO-v} follows from
Lemma~\ref{lem:convex-distance-regularity}. Convexity also gives
$\II_{\partial\mathcal C}\le0$ almost everywhere for the outward normal.

\begin{theorem}\label{ThHausdorff}
Assume that $M$ admits barriers \textup{(B1)--(B3)}. Let
\[
 0\le\lambda\le\lambda_*:=\frac{(n-1)^2\kappa}{4},\qquad
 s_\lambda:=\frac{1}{\sqrt a}
 \left(\frac{(n-1)\sqrt\kappa}{2}+\sqrt{\lambda_*-\lambda}\right).
\]
If a domain $\Omega\subset M$ satisfies $\mathcal H^{s_\lambda}_{d_o}(\partial_\infty\Omega)=0$, and $f\in C(\mathbb R)$ satisfies
\begin{equation}\label{eq:sign-nonlinearity}
 t f(t)\le\lambda t^2\qquad t\in\mathbb R,
\end{equation}
then every bounded classical solution
$u\in C^2(\Omega)\cap C^0(\overline\Omega)$ of
\begin{equation}\label{Dirichlet}
 \begin{cases}
 \Delta u+f(u)=0&\text{in }\Omega,\\
 u=0&\text{on }\partial\Omega
 \end{cases}
\end{equation}
is identically zero. The same assertion holds for bounded viscosity
solutions in $C^0(\overline\Omega)$ with these boundary values.
\end{theorem}

\begin{proof}
Suppose that $u$ is a nonzero bounded viscosity solution of~\eqref{Dirichlet}.
If it has no positive value, replace $(u,f)$ by $(\widetilde u,\widetilde f)=(-u,-f(-\,\cdot))$. 
The transformed pair satisfies the same hypotheses, since
$t\widetilde f(t)=(-t)f(-t)\le\lambda t^2$. We may therefore fix
$o\in\Omega$ with $u(o)>0$. Put $\Lambda=2\sup_\Omega u>0$.
On $\{u>0\}$, condition~\eqref{eq:sign-nonlinearity} gives
$f(u)\le\lambda u$, so $u$ is a viscosity subsolution of
$\Delta u+\lambda u=0$ there.

For later use, if $D$ is bounded and supports a positive continuous viscosity
subsolution $w$ of $\Delta w+\lambda w=0$ which vanishes on $\partial D$,
choose concentric geodesic balls with
$\overline D\subset B_R\Subset B_{R+1}$. Proposition~\ref{prop:Lambda},
strict domain monotonicity on these smooth balls and McKean's inequality
$\lambda_0(M)\ge\lambda_*$~\cite{McKean} imply
\begin{equation}\label{eq:spectral-contradiction}
 \lambda\ge\lambda_1(B_R)>\lambda_1(B_{R+1})
          \ge\lambda_0(M)\ge\lambda_*\ge\lambda,
\end{equation}
a contradiction. This argument requires no regularity of $\partial D$.

Set $X=\partial_\infty\Omega$. If $X$ is empty, compactness of the cone
compactification implies that $\Omega$ is bounded. The component of
$\{u>0\}$ containing $o$ then gives the contradiction
\eqref{eq:spectral-contradiction} with $w=u$. Hence assume that $X$ is
nonempty. It is a closed subset of the compact ideal boundary and is
therefore compact.

Let $h,t_0,C_h$ be as in Lemma~\ref{lem:fast-profile}, and write
$\beta=\beta_+(\lambda)$, so that $s_\lambda=\beta/\sqrt a$.
Choose $r_*>0$ small enough that, for $0<r<r_*$,
\[
r<r_{\mathrm v},\qquad
 \frac1{\sqrt a}\log\frac{C_1}{r}-C_2>t_0.
\]
Given $\varepsilon>0$, the hypothesis $\mathcal H^{s_\lambda}_{d_o}(X)=0$
provides a cover of $X$ by open visual balls $B_{d_o}(\xi_i,r_i)$ with
$r_i<r_*$ and $\sum_i r_i^{s_\lambda}<\varepsilon$.
Indeed, $X$ has a countable cover by sets $E_i$, $i\ge1$, meeting $X$, with $\operatorname{diam}_{d_o}E_i<r_*/2$ and $\sum_i(2\operatorname{diam}_{d_o}E_i)^{s_\lambda}<\varepsilon/2$. If $\operatorname{diam}_{d_o}E_i>0$, then $E_i$ lies in the open ball of radius $r_i=2\operatorname{diam}_{d_o}E_i$ centered at any $\xi_i\in E_i$. Otherwise $E_i=\{\xi_i\}$, and we take $0<r_i<r_*$ with $r_i^{s_\lambda}<2^{-i-1}\varepsilon$.
By compactness, retain a finite subcover, indexed by $1\le i\le m$;
its sum still satisfies the same bound.
For each visual ball choose its barrier $\mathcal C_i$ and the
half-space $H_{q_i}(\nu_i)$ from \textup{(B2)}, and write
$d_i=d_{\mathcal C_i}$.
Define the open set and finite sum
\[
 U=\Omega\cap\bigcap_{i=1}^m\{d_i>t_0\},\quad
 v_i=\frac{\Lambda}{h(t_0)}h(d_i),\quad
 v=\sum_{i=1}^mv_i\quad\text{on }U.
\]
Our choice of $r_*$ ensures $o\in U$. Each $v_i$ is positive, belongs to
$C^{1,1}_{\mathrm{loc}}(U)$, and satisfies
$\Delta v_i+\lambda v_i\le0$ almost everywhere, by Lemma~\ref{EDO-v}.
The same differential inequality and regularity hold for their finite sum.
Moreover, (B3) and the profile estimate yield
\[
 v_i(o)\le\frac{\Lambda C_h}{h(t_0)}
          \exp\!\left[-\beta\left(
             \frac1{\sqrt a}\log\frac{C_1}{r_i}-C_2\right)\right]
       =Ar_i^{s_\lambda},
\]
where $A=\frac{\Lambda C_h}{h(t_0)}e^{\beta C_2}C_1^{-s_\lambda}$ is
independent of the cover. Consequently
\[
 v(o)\le A\sum_{i=1}^m r_i^{s_\lambda}<A\varepsilon.
\]
We may therefore fix $\varepsilon=u(o)/A$ before choosing the cover,
and then $v(o)<u(o)$.

We claim that $U$ is bounded. Otherwise, a divergent sequence $x_j\in U$
has a subsequence converging in the cone compactification to some $\eta\in X$.
Choose a covering ball containing $\eta$. By (B2),
\(
 \measuredangle_{q_i}(\nu_i,\operatorname{Hor}_{q_i}^{-1}(\eta))<\pi/2.
\)
Continuity of geodesic directions in the cone topology gives
$x_j\in H_{q_i}(\nu_i)\subset\mathcal C_i$ for all sufficiently large $j$,
contrary to $d_i(x_j)>t_0$. This proves the claim.

The functions $v_i$ extend continuously to $\overline U$, since
$d_i\ge t_0$ there. At a point of $\partial U\cap\partial\Omega$,
$u=0<v$. At any other point of $\partial U$, at least one constraint
satisfies $d_i=t_0$, and therefore
$v\ge v_i=\Lambda>u$. Thus $u-v<0$ on $\partial U$.
Let $D$ be the component of $\{u-v>0\}\cap U$ containing $o$.
Continuity and compactness give $\overline D\Subset U$;
in particular $w=u-v$ is continuous on $\overline D$, positive in $D$,
and zero on $\partial D$. Since $u>v>0$ in $D$,
Lemma~\ref{lem:difference} shows that $w$ is a viscosity subsolution of
$\Delta w+\lambda w=0$ there. We reach
\eqref{eq:spectral-contradiction}. Classical solutions are viscosity
solutions, so both conclusions follow.
\end{proof}

Figure~\ref{fig:truncation} illustrates the truncation and the choice of $D$
when the ideal boundary consists of a single point.
\begin{figure}[htbp]
  \centering
  \includegraphics[width=0.40\linewidth]{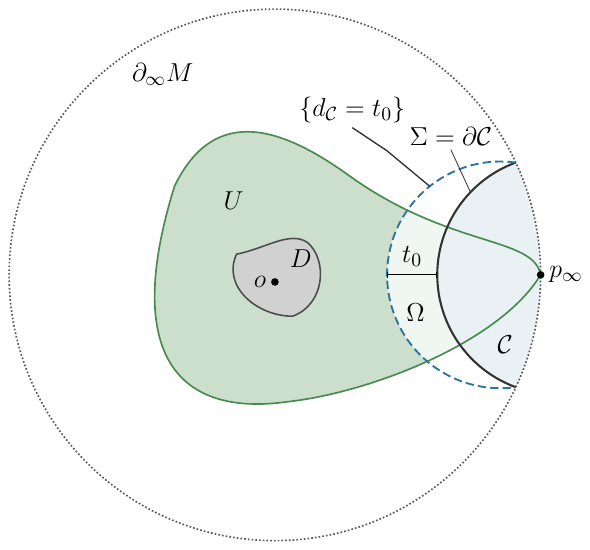}

  \caption{\small Schematic picture for $X=\{p_\infty\}$ and a single barrier
  $\mathcal C$, with $\Sigma=\partial\mathcal C$. Cutting at distance $t_0$
  gives the bounded set $U=\Omega\cap\{d_{\mathcal C}>t_0\}$.
  The component $D$ of $\{u-v>0\}\cap U$ containing $o$ is compactly
  contained in $U$. The contours are schematic: $U$ need not be connected,
  and no boundary regularity of $\Omega$, $U$ or $D$ is used.}
  \label{fig:truncation}

\end{figure}

\begin{remark}\label{rem:pinched-barriers}
In the pinched setting $-\kappa_1\le K\le-\kappa_2<0$,
Corollary~\ref{cor:visual-barriers} supplies \textup{(B1)--(B3)} with
$\kappa=a=\kappa_2$, $C_1=1$ and $C_2=C_0$.
Theorem~\ref{thm:main} follows. In particular,
\[
 \dim_H(\partial_\infty\Omega)<\frac{n-1}{2}
\]
is sufficient for every $0\le\lambda\le(n-1)^2\kappa_2/4$.
The lower curvature bound enters the uniform barrier constants, but not
this dimension threshold for the fixed visual parameter $\sqrt{\kappa_2}$.
\end{remark}

\begin{remark}[Role of the lower curvature bound]\label{rem:lower-curvature}
The lower curvature bound $K\ge-\kappa_1$ enters the proof of
Theorem~\ref{thm:main} only through Theorem~\ref{thm:Anderson}, which supplies
a separation constant $r_0$ uniform over the moving centers used in
Corollary~\ref{cor:visual-barriers}.
Sections~\ref{Preliminaries}, \ref{Distance} and~\ref{Main} use only
$K\le-\kappa_2<0$, together with the barrier assumptions
\textup{(B1)--(B3)} in Section~\ref{Main}.

In dimension $n=2$, Theorem~\ref{thm:main} holds under the sole curvature
assumption $K\le-\kappa_2<0$. Indeed, for $p\in M$ and $w\in S_pM$,
the closed half-plane $\mathcal C=\overline{H_p(-w)}$ is bounded by the
complete geodesic $\exp_p(w^\perp)$ and is convex. A geodesic segment
leaving this half-plane would meet its boundary geodesic twice,
contradicting uniqueness of geodesic segments. Its boundary is totally
geodesic, so $\II=0$. For every $y\in\mathcal C\setminus\{p\}$ the angle
at $p$ with $w$ is at least $\pi/2$; hinge comparison gives
\[
 d(\gamma_{p,w}(t),y)^2\ge t^2+d(p,y)^2,
 \qquad d(\gamma_{p,w}(t),\mathcal C)=t\quad(t\ge0).
\]
Thus the properties of Theorem~\ref{thm:Anderson} hold with $r_0=0$,
and the visual-ball construction and the proof in Section~\ref{Main}
apply unchanged.

For $n\ge3$, the weaker curvature assumptions of
Ripoll--Telichevesky~\cite{RipollTelichevesky} yield the
qualitative convexity at infinity needed for their Dirichlet problem.
Their construction, however, does not provide the separation estimate
with $r_0$ uniform over the moving base points used in
Corollary~\ref{cor:visual-barriers}. In particular, transferring an
exponential lower curvature bound from one center to another may change
its multiplicative constant; changing only the starting radius does
not in general preserve the same bound. Consequently the estimate
\textup{(B3)} used here is not justified by that construction under
those assumptions. We do not know whether Theorem~\ref{thm:main}
extends to that setting.
\end{remark}

\appendix
\section{Proof of Theorem \ref{thm:Anderson}}\label{appendixA}

We follow the convex-set construction of Anderson \cite{Anderson83} and
Ripoll--Telichevesky \cite{RipollTelichevesky}. We include the induction,
since the distance estimate in Theorem~\ref{thm:Anderson} requires that the
removed regions remain excluded at every subsequent step. Throughout this
appendix, $\rho_y(x):=d(y,x)$, $B_y(r)$ denotes an open ball, and
$\overline B_y(r)$ denotes its closure.

\begin{lemma}\label{lemmaA1}
Let $M^n$ be a Cartan--Hadamard manifold with
$-\kappa_1\le K\le-\kappa_2<0$. Fix a smooth nondecreasing function
$\phi:[0,\infty)\to[0,1]$ such that $\phi=0$ on $[0,1/2]$ and
$\phi=1$ on $[1,\infty)$, and choose $L>0$ with
$|\phi'|,|\phi''|\le L$. Set $f_q=\phi\circ\rho_q$.
There exists $\varepsilon_0=\varepsilon_0(\kappa_1,\kappa_2,L)>0$ such that
\[
 \{x:\rho_p(x)-\varepsilon f_q(x)\le R\}
\]
is a closed strictly convex set with smooth boundary, for every $p,q\in M^n$,
$R>0$, and $0<\varepsilon<\varepsilon_0$.
\end{lemma}

\begin{proof}
Since $\phi$ is constant near zero, $f_q$ is smooth also at $q$. Hessian
comparison, under the lower bound $K\ge-\kappa_1$, and the supports of $\phi'$ and $\phi''$ give
\[
 |\nabla f_q|\le L,\quad
 |\nabla^2f_q(W,W)|\le A|W|^2,
 \quad
 A=L\bigl(1+\sqrt{\kappa_1}\coth(\sqrt{\kappa_1}/2)\bigr).
\]

On $M^n\setminus\{p\}$, the upper curvature bound gives
\begin{equation}\label{Hessiancomparison}
 \nabla^2\rho_p(W,W)\ge
 \sqrt{\kappa_2}\bigl(|W|^2-\langle\nabla\rho_p,W\rangle^2\bigr).
\end{equation}

Let $G=\rho_p-\varepsilon f_q$. If $dG(W)=0$, then
$|\langle\nabla\rho_p,W\rangle|\le\varepsilon L|W|$, and therefore
\[
 \nabla^2G(W,W)\ge
 \bigl(\sqrt{\kappa_2}(1-\varepsilon^2L^2)-\varepsilon A\bigr)|W|^2.
\]

Choose $\varepsilon_0\le
 \min\left\{\frac1{2L},\frac{\sqrt{\kappa_2}}{2A}\right\}$. The last coefficient is at least $\sqrt{\kappa_2}/4$.
Moreover, $|\nabla G|\ge1-\varepsilon L\ge1/2$ away from $p$.
Since $G(p)\le0<R$, the level $G=R$ is smooth.

To verify global convexity, join two distinct points of $\{G\le R\}$ by a geodesic
segment. If an interior point had $G\ge R$, the restriction of $G$ to the
segment would attain an interior maximum of value at least $R$.
That point is different from $p$. Its first derivative vanishes and its
second derivative is strictly positive by the preceding estimate, a
contradiction. Thus the open segment lies in $\{G<R\}$.
Closedness follows from continuity of $G$.
\end{proof}

Fix this cutoff and set
$\varepsilon=\frac12\min\{\varepsilon_0,1/2\}$. In particular,
$0<\varepsilon<1/2$ depends only on the curvature bounds.
Let $r_0>2$, to be chosen below, and write $r_k=r_0+k\varepsilon$.
\rev{Given $p$ and $w$ as in Theorem~\ref{thm:Anderson}, we} start with
\[
 C_0=\overline B_p(r_0),\quad
 T_0=C_p(w,\pi/4)\cap S_p(r_0),\quad D_0=\varnothing.
\]

For $q\in T_0$, set
\begin{equation}\label{eq:C1}
 C_{1,q}=\{x:\rho_p(x)-\varepsilon f_q(x)\le r_0\},
 \quad C_1=\widetilde C_1=\bigcap_{q\in T_0}C_{1,q}.
\end{equation}
For $k\ge1$, once $C_k$ has been constructed, define
\begin{equation}\label{eq:inductive_definitions}
 \begin{split}
 D_k&=\overline B_p(r_k)\setminus C_k,\qquad
 T_k=\overline{S_p(r_k)\setminus\partial C_k},\\
 \widetilde C_{k+1}&=\bigcap_{q\in T_k}C_{k+1,q},\qquad
 C_{k+1}=\widetilde C_{k+1}\setminus D_k,
 \end{split}
\end{equation}
where the closure defining $T_k$ is taken in $S_p(r_k)$ and
\begin{equation}\label{eq:Ckq}
 C_{k+1,q}=\{x:\rho_p(x)-\varepsilon f_q(x)\le r_k\}.
\end{equation}

\rev{Figure~\ref{fig:appendix-induction} separates the new constraints from the removal of the region already excluded at the preceding step.}
\begin{figure}[htbp]
  \centering
  \includegraphics[width=\linewidth]{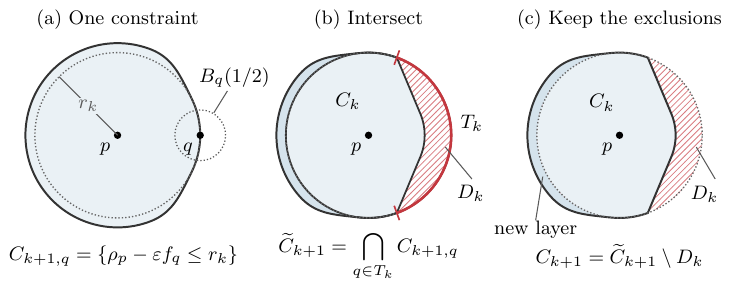}
  \begingroup\ifrevfinal\else\color{blue}\fi
  \caption{\small One step of the construction, schematically; distances and transition regions are not to scale. (a) A single constraint $C_{k+1,q}$ contains $\overline B_p(r_k)$ and coincides with it on $B_q(1/2)$, where $q\in T_k$ is the cutoff center. (b) The arc $T_k$ lies between the two ticks. The intersection $\widetilde C_{k+1}$ still contains $\overline B_p(r_k)$, including the hatched set $D_k=\overline B_p(r_k)\setminus C_k$. (c) Removing $D_k$ gives $C_{k+1}\cap\overline B_p(r_k)=C_k$, and the only new points form an outer layer in $\overline B_p(r_k+\varepsilon)\setminus\overline B_p(r_k)$. The hatched region, including its points on $S_p(r_k)$, stays excluded.}
  \label{fig:appendix-induction}
  \endgroup
\end{figure}

We next show that these sets are well defined, closed, and convex. 
Let $K\subset\overline B_p(R)$ be closed and convex, with $p\in K$.
Suppose that $T=\overline{S_p(R)\setminus\partial K}$ is nonempty, and set

\[
 L=\bigcap_{q\in T}\{\rho_p-\varepsilon f_q\le R\},\quad
 K^+=L\setminus(\overline B_p(R)\setminus K).
\]

By Lemma~\ref{lemmaA1}, $L$ is closed and convex, and
\begin{equation}\label{eq:closedball-stabilization}
 \overline B_p(R)\subset L\subset\overline B_p(R+\varepsilon),
 \quad K^+\cap\overline B_p(R)=K.
\end{equation}
For every $q\in T$, the constraint centered at $q$ and $f_q=0$ on
$B_q(1/2)$ give
\begin{equation}\label{eq:seam-local}
 L\cap B_q(1/2)\subset\overline B_p(R),\quad
 K^+\cap B_q(1/2)=K\cap B_q(1/2).
\end{equation}

We first prove that $K^+$ is closed. In $B_p(R)$ it agrees locally with
$K$, and outside $\overline B_p(R)$ it agrees locally with $L$.
Identity~\eqref{eq:seam-local} settles points of $T$.
If $x\in S_p(R)\setminus T$, a relatively open spherical neighborhood of
$x$ lies in $\partial K\subset K$. Convexity and $p\in K$ imply that all
radial segments from $p$ to that neighborhood belong to $K$.
Thus there is a neighborhood $V$ of $x$ with
$V\cap\overline B_p(R)\subset K$, and $K^+\cap V=L\cap V$.
These local identities prove closedness. \rev{The two local descriptions at the sphere are illustrated in Figure~\ref{fig:appendix-closedness}.}
\begin{figure}[!htbp]
  \centering
  \includegraphics[width=\linewidth]{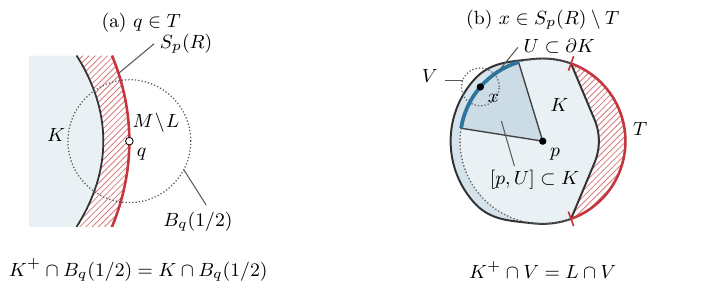}
  \begingroup\ifrevfinal\else\color{blue}\fi
  \caption{\small The two cases in the proof that $K^+$ is closed (local incidence diagrams, not metric models). Hatching marks the removed set $\overline B_p(R)\setminus K$. (a) At $q\in T$ the cutoff vanishes on $B_q(1/2)$, so $L\cap B_q(1/2)\subset\overline B_p(R)$ and $K^+$ agrees with $K$ there. (b) At $x\in S_p(R)\setminus T$, a spherical neighborhood $U$ of $x$ lies in $\partial K$, and convexity puts the radial region $[p,U]$ in $K$. A small neighborhood $V$ of $x$ then satisfies $V\cap\overline B_p(R)\subset K$, hence $K^+\cap V=L\cap V$.}
  \label{fig:appendix-closedness}
  \endgroup
\end{figure}

To prove convexity, let $x,y\in K^+$ and let $\sigma:[0,1]\to L$ be the
geodesic segment joining them. The set
$\sigma^{-1}(\overline B_p(R))$ is either empty or a closed interval
$[a,b]$, possibly reduced to one point. Its endpoints lie in $K$.
Indeed, this is immediate if $a=0$ or $b=1$. If an endpoint corresponding
to a crossing of $S_p(R)$ did not belong to $K$, it would belong to
$S_p(R)\setminus K=S_p(R)\setminus\partial K\subset T$.
Identity~\eqref{eq:seam-local} would then exclude the neighboring points
of $\sigma$ outside $\overline B_p(R)$, which is impossible.
The same argument applies when $a=b$.
Since $K$ is convex, $\sigma([a,b])\subset K$, and the remaining parts of
$\sigma$ lie in $L\setminus\overline B_p(R)$. Therefore
$\sigma([0,1])\subset K^+$. \rev{Figure~\ref{fig:appendix-convexity} shows how the local constraint at the sphere keeps the joining geodesic inside $K^+$.}
\begin{figure}[!htbp]
  \centering
  \includegraphics[width=\linewidth]{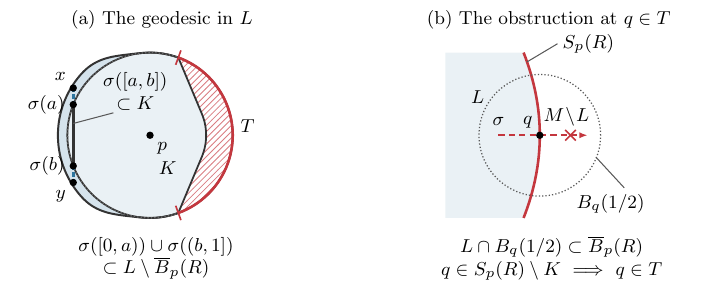}
  \begingroup\ifrevfinal\else\color{blue}\fi
  \caption{\small Convexity of $K^+$. (a) The part $\sigma([a,b])$ of the segment in $\overline B_p(R)$ joins two points of $K$, so it lies in $K$; the remaining parts lie in $L\setminus\overline B_p(R)$. (b) If a crossing point $q$ were not in $K$, then $q\in T$ and $L\cap B_q(1/2)\subset\overline B_p(R)$, so the points of $\sigma$ just outside $\overline B_p(R)$ would not lie in $L$. The drawing shows $a<b$; the same argument applies when $a=b$.}
  \label{fig:appendix-convexity}
  \endgroup
\end{figure}

The first set $C_1$ is closed and convex by Lemma~\ref{lemmaA1}, with $ \overline B_p(r_0)\subset C_1\subset\overline B_p(r_1)$. Write $\gamma(t)=\exp_p(tw)$ and $z_0=\gamma(r_0)\in T_0$.
The constraint centered at $z_0$ excludes every $\gamma(t)$ with
$r_0<t\le r_1$, since $t-r_0\le\varepsilon<1/2$.
No point $\gamma(t)$ with $t>r_1$ belongs to $C_1$ either.
Suppose inductively, for some $k\ge1$, that $C_k$ is closed and convex, $\overline B_p(r_0)\subset C_k\subset\overline B_p(r_k)$, and excludes
$\gamma((r_0,\infty))$. Then $\gamma(r_k)\in T_k$, so $T_k$ is nonempty.
The preceding argument applies with $K=C_k$ and $R=r_k$.
It proves that $C_{k+1}$ is closed and convex, and gives
\begin{equation}\label{eq:local-stabilization}
 C_k\subset C_{k+1}\subset\overline B_p(r_{k+1}),\quad
 C_{k+1}\cap\overline B_p(r_k)=C_k.
\end{equation}

The interval $\gamma((r_0,r_k])$ remains excluded by this identity.
If any later point of the axis belonged to $C_{k+1}$, convexity and
$p\in C_{k+1}$ would force that interval into $C_{k+1}$.
Thus the whole open axis remains excluded, completing the induction. \rev{Figure~\ref{fig:appendix-axis} 
shows why the excluded axial interval stays excluded at every later step.}
\begin{figure}[!htbp]
  \centering
  \includegraphics[width=0.6\linewidth]{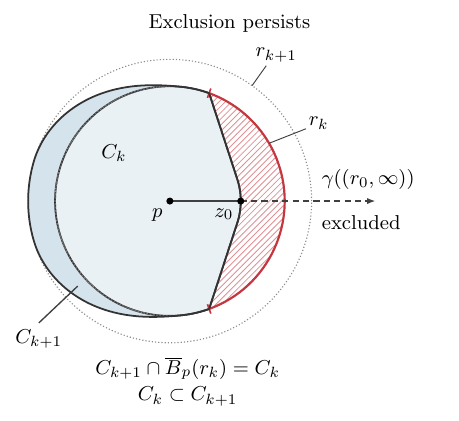}
  \begingroup\ifrevfinal\else\color{blue}\fi
  \caption{\small Persistent exclusion of the axis (schematic section). The equality $C_{k+1}\cap\overline B_p(r_k)=C_k$ keeps $\gamma((r_0,r_k])$ outside $C_{k+1}$. Since $p\in C_{k+1}$ and $C_{k+1}$ is convex, no later point of the ray lies in $C_{k+1}$ either.}
  \label{fig:appendix-axis}
  \endgroup
\end{figure}

For the angular estimate, define
\begin{equation}\label{eq:theta0}
 \theta_0=\sup\{\measuredangle_p(\gamma'_{p,q}(0),\gamma'_{p,x}(0)):
 q\in T_0,\ x\in\overline B_q(1)\},
\end{equation}
and, for $k\ge1$, define
\begin{equation}\label{eq:thetak}
 \theta_k=\sup\{\measuredangle_p(\gamma'_{p,q}(0),\gamma'_{p,x}(0)):
 q\in T_k,\ x\in\overline B_q(1)\}.
\end{equation}
Set $\alpha_k=\pi/4+\sum_{j=0}^{k-1}\theta_j$ for $k\ge1$.
We claim that
\begin{equation}\label{eq:angular_control}
 \measuredangle_p(w,\gamma'_{p,x}(0))\le\alpha_k
 \qquad x\in D_k.
\end{equation}

For $x\in D_1$, some $q\in T_0$ satisfies
$\rho_p(x)-\varepsilon f_q(x)>r_0$.
Since $\rho_p(x)\le r_0+\varepsilon$, we have $f_q(x)<1$, hence
$d(q,x)<1$. The claim follows from the triangle inequality for angles.
For the inductive step, observe that
\(
 D_{k+1}=D_k\cup
 \bigl(\overline B_p(r_{k+1})\setminus\widetilde C_{k+1}\bigr).
\)
Moreover, $S_p(r_k)\setminus\partial C_k=S_p(r_k)\setminus C_k\subset D_k$,
so every point of its closure $T_k$ has angle at most $\alpha_k$ with $w$.
For a point in the second set of the preceding union, one of the defining
constraints fails. The same calculation gives $d(q,x)<1$ for some $q\in T_k$.
Thus its angle with $w$ is at most $\alpha_k+\theta_k$, proving the claim.
In particular, the finite-stage inclusion is
\begin{equation}\label{eq:truncated-cone}
 C_p(-w,\pi-\alpha_k)\cap\overline B_p(r_k)\subset C_k,
\end{equation}
whenever $\alpha_k<\pi$.

Let $b=\sqrt{\kappa_2}$. Comparison with the hyperbolic plane of curvature
$-\kappa_2$ \cite[Proposition~II.1.7(4)]{BridsonHaefliger} gives
\begin{equation}\label{eq:visual-angle-decay}
 \theta_k\le
 \arcsin\!\left(\frac{\sinh b}{\sinh(br_k)}\right)
 \le A_0e^{-br_k},\qquad r_k>2 ,
\end{equation}
where $A_0$ depends only on $\kappa_2$.
Indeed, for $q\in S_p(r_k)$ and $d(q,x)\le1$, the angle at $p$ is at most
the angle of the hyperbolic triangle with the same side lengths.
That comparison angle is acute, since its opposite side is shorter than
both other sides. The hyperbolic sine rule bounds its sine by
$\sinh b/\sinh(br_k)$. Degenerate triangles follow by continuity.
It follows that
\[
 \sum_{k=0}^\infty\theta_k\le
 \frac{A_0e^{-br_0}}{1-e^{-b\varepsilon}}.
\]

We choose $r_0\ge R^*(\kappa_1,\kappa_2)>2$ so that
$\alpha_\infty:=\pi/4+\sum_{k=0}^\infty\theta_k<\pi/2$. Set $\mathcal C=\bigcup_{k\ge0}C_k$.
The sets are increasing and convex, so $\mathcal C$ is convex.
It is also closed. In fact, given any bounded neighborhood, choose $k$
with that neighborhood contained in $B_p(r_k)$.
By~\eqref{eq:local-stabilization}, every subsequent set agrees there with
the closed set $C_k$. Thus $\mathcal C$ is locally equal to a closed set.
For any point of $C_p(-w,\pi-\alpha_\infty)$, choose $k$ with
$r_k$ larger than its distance to $p$ and apply~\eqref{eq:truncated-cone}.
We obtain $C_p(-w,\pi-\alpha_\infty)\subset\mathcal C$. \rev{Figure~\ref{fig:appendix-angle} separates the angular estimate, the local stabilization that makes $\mathcal C$ closed, and the contact at $z_0$ used in Lemma~\ref{lemmaA2}(ii).}
\begin{figure}[!htbp]
  \centering
  \includegraphics[width=\linewidth]{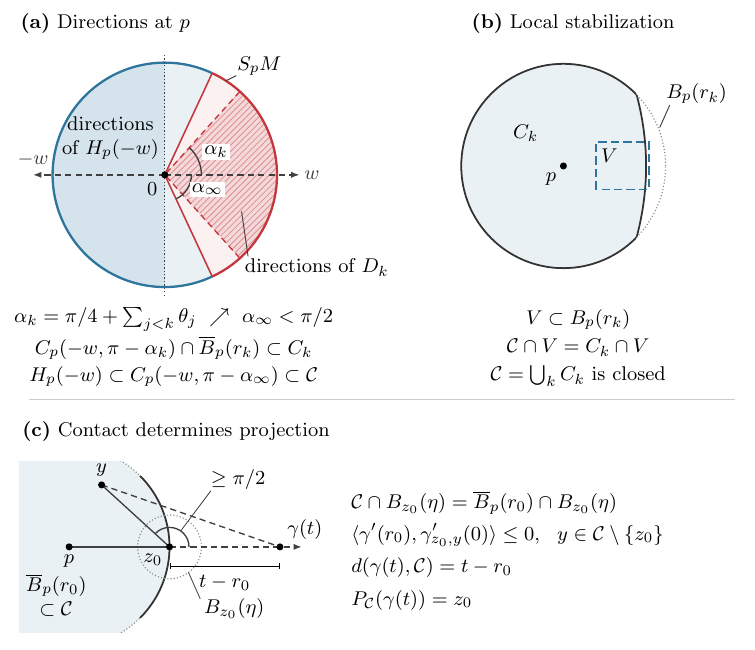}
  \begingroup\ifrevfinal\else\color{blue}\fi
  \caption{\small Angular control, passage to the limit, and the contact at $z_0$. (a) Directions at $p$ in a two-dimensional section of $T_pM$; the circle stands for unit directions and the sectors only record angles, so this is not a picture of $\mathcal C$. The hatched cap of half-aperture $\alpha_k$ contains the directions of $D_k$, and all these caps lie in the cap of half-aperture $\alpha_\infty<\pi/2$. The remaining directions, those of $C_p(-w,\pi-\alpha_\infty)$, include the hemisphere of directions of $H_p(-w)$. This cone has half-aperture greater than $\pi/2$; only $\mathcal C$ is claimed to be convex. At stage $k$ the cone inclusion holds only in $\overline B_p(r_k)$. (b) If $V\subset B_p(r_k)$, every later set agrees with $C_k$ on $V$, so $\mathcal C\cap V=C_k\cap V$; this local stabilization makes $\mathcal C$ closed. (c) Near $z_0=\gamma(r_0)$ the set $\mathcal C$ agrees with $\overline B_p(r_0)$; the solid arc is a piece of $\partial\mathcal C$ on $S_p(r_0)$, and the dotted arc is $S_p(r_0)$ where $\mathcal C$ may be larger. Every segment from $z_0$ to $y\in\mathcal C\setminus\{z_0\}$ therefore makes an angle at least $\pi/2$ with the outgoing axis, and hinge comparison gives $d(\gamma(t),\mathcal C)=t-r_0$ and $P_{\mathcal C}(\gamma(t))=z_0$. No supporting hypersurface is constructed.}
  \label{fig:appendix-angle}
  \endgroup
\end{figure}

\begin{lemma}\label{lemmaA2}
For the above choice of $r_0$, the set $\mathcal C$ is closed and convex,
and satisfies
\begin{itemize}
\item[\rm(i)] $H_p(-w)\subset\mathcal C$;
\item[\rm(ii)] $d(\gamma_{p,w}(t),\mathcal C)=t-r_0$ for every $t\ge r_0$.
\end{itemize}
\end{lemma}

\begin{proof}
Item {\rm (i)} follows from $\alpha_\infty<\pi/2$ and the cone inclusion.
For item {\rm (ii)}, the constraint centered at $z_0=\gamma(r_0)$ gives $C_1\cap B_{z_0}(1/2)=\overline B_p(r_0)\cap B_{z_0}(1/2)$. Choose $0<\eta<\min\{1/2,\varepsilon\}$.
Since $B_{z_0}(\eta)\subset B_p(r_1)$, local stabilization yields $ \mathcal C\cap B_{z_0}(\eta)
 =\overline B_p(r_0)\cap B_{z_0}(\eta)$. For $y\in\mathcal C\setminus\{z_0\}$, the segment $[z_0,y]$ lies in
$\mathcal C$. Differentiating $\rho_p$ at its initial point therefore gives
\begin{equation}\label{eq:support}
 \langle\gamma'(r_0),\gamma'_{z_0,y}(0)\rangle\le0.
\end{equation}

The Cartan--Hadamard hinge comparison \cite[Proposition~II.1.7(5)]{BridsonHaefliger} now implies
\[
 d(\gamma(t),y)^2\ge(t-r_0)^2+d(z_0,y)^2
 \ge(t-r_0)^2 ,\qquad t\ge r_0.
\]
Since $z_0\in\mathcal C$, the reverse distance inequality is immediate.
In particular, the metric projection satisfies
$P_{\mathcal C}(\gamma(t))=z_0$ for every $t\ge r_0$. \rev{The contact angle and the 
resulting projection are shown in Figure~\ref{fig:appendix-angle}(c).}
\end{proof}

We finish by passing to a parallel convex body. We use the local
$C^{1,1}$ distance regularity of Lemma~\ref{lem:convex-distance-regularity}
to obtain convex parallel hypersurfaces with nonpositive second fundamental form.

\begin{lemma}\label{lem:parallel}
Let $M^n$ be a Cartan--Hadamard manifold and let
$\mathcal C\subsetneq M^n$ be nonempty, closed, and
convex. Set $d=d(\cdot,\mathcal C)$ and, for $\delta>0$, let $ \mathcal C_\delta=\{d\le\delta\}$ 
and $\Sigma_\delta=\partial\mathcal C_\delta$.

Then, $d\in C^{1,1}_{\mathrm{loc}}(M^n\setminus\mathcal C)$,
$\mathcal C_\delta$ is closed and convex, and $\Sigma_\delta=\{d=\delta\}$
is a properly embedded $C^{1,1}$ hypersurface.
Its signed distance, positive outside $\mathcal C_\delta$, satisfies 
$ d_{\Sigma_\delta}=d-\delta $ on $M^n\setminus\mathcal C$.

In particular, it is $C^{1,1}_{\mathrm{loc}}$ on a neighborhood of
$\{d\ge\delta/2\}$. With respect to the outward normal $N=\nabla d$,
\[
 \II^N_{\Sigma_\delta}\le0
 \qquad\hbox{a.e. on }\Sigma_\delta.
\]
\end{lemma}

\begin{proof}
By Lemma~\ref{lem:convex-distance-regularity},
$d\in C^{1,1}_{\mathrm{loc}}(M^n\setminus\mathcal C)$ and $|\nabla d|=1$ there.
Let $P=P_{\mathcal C}$ be the metric projection introduced in Section~\ref{Distance}.
Completeness gives existence of a nearest point, and strict convexity of
squared distance along geodesics gives uniqueness.
We recall that $P$ is nonexpansive and $d$ is convex in a Cartan--Hadamard manifold~\cite[Proposition~II.2.4(4) and Corollary~II.2.5(1)]{BridsonHaefliger}.

For $x,y\in M^n$, write $\log_x(y)=\exp_x^{-1}(y)\in T_xM^n$.
The Cartan--Hadamard theorem makes this inverse well defined on all of $M^n$.
Thus $\log_x(P(x))$ is the vector at $x$ pointing towards its nearest point,
of length $d(x)$. First variation gives, outside $\mathcal C$,
\begin{equation}\label{eq:projection-gradient}
 \nabla d(x)=-\frac{\log_x(P(x))}{d(x)},\qquad |\nabla d(x)|=1.
\end{equation}
The minus sign in~\eqref{eq:projection-gradient} makes $\nabla d$
point away from $\mathcal C$. Convexity of $d$ implies convexity of $\mathcal C_\delta$, and
$|\nabla d|=1$ shows that $\Sigma_\delta$ is a closed regular level of a
$C^{1,1}$ function. It is therefore a properly embedded $C^{1,1}$
hypersurface. 

At almost every point of $\Sigma_\delta$ its local graph is twice
differentiable. At such a point $x$, use normal coordinates centered at
$x$, with the outward normal as the last coordinate direction. The
second fundamental form $\II^N$ at $x$ is the Hessian of the graph
function. Geodesic segments from $x$ to nearby points of
$\Sigma_\delta$ lie in $\mathcal C_\delta$ by convexity. Expanding this
inclusion to second order gives $\II^N\le0$ at $x$.

Finally, a ray starting at $q=P(x)$ and passing through $x$ continues to
have projection $q$. Indeed, first variation at $q$ gives an angle at
least $\pi/2$ with every segment from $q$ to $\mathcal C$~\cite[Proposition~II.2.4(3)]{BridsonHaefliger}, and hinge
comparison~\cite[Proposition~II.1.7(5)]{BridsonHaefliger} proves that its point at parameter $t$ has distance $t$ from
$\mathcal C$. The point of this ray at parameter $\delta$ belongs to
$\Sigma_\delta$ and is at distance $|d(x)-\delta|$ from $x$.
The reverse inequality follows from the fact that $d$ is $1$-Lipschitz.
This proves the signed-distance identity and also shows that
$\mathcal C_\delta$ is proper. \rev{Figure~\ref{fig:appendix-parallel} illustrates the 
projection ray and the sign of the distance to the parallel hypersurface.}
\end{proof}

\begin{figure}[!htbp]
  \centering
  \includegraphics[width=0.8\linewidth]{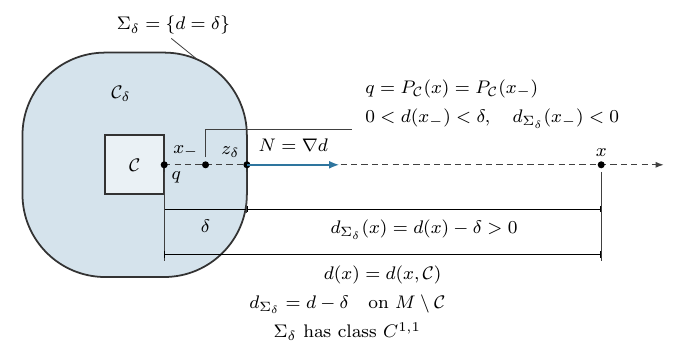}
  \begingroup\ifrevfinal\else\color{blue}\fi
  \caption{\small The parallel body and its signed distance, illustrated by a Euclidean convex set, a special case of $K\le0$. Along the outward ray from $q=P_{\mathcal C}(x)$, the projection stays equal to $q$. The point $z_\delta$ at distance $\delta$ from $q$ belongs to $\Sigma_\delta$; the normal $N=\nabla d$ points outward. For $x_-$ between $q$ and $z_\delta$, $d_{\Sigma_\delta}(x_-)<0$; beyond $z_\delta$ it is positive. In both cases $d_{\Sigma_\delta}=d-\delta$. The rounded parallel boundary has class $C^{1,1}$; no further smoothing is assumed.}
  \label{fig:appendix-parallel}
  \endgroup
\end{figure}

Apply Lemma~\ref{lem:parallel} to the set of Lemma~\ref{lemmaA2} and fix,
for instance, $\delta=1$. Since $\mathcal C\subset\mathcal C_\delta$, property {\rm(i)} is preserved.
For $d(x,\mathcal C)\ge\delta$, the projection segment and the
$1$-Lipschitz property of $d$ give $d(x,\mathcal C_\delta)=d(x,\mathcal C)-\delta$. Consequently,
$ d(\gamma_{p,w}(t),\mathcal C_\delta)=t-(r_0+\delta)$, $t\ge r_0+\delta $. The boundary 
is $C^{1,1}$ and its outward second fundamental form is
nonpositive almost everywhere. Renaming $r_0+1$ as $r_0$ proves
Theorem~\ref{thm:Anderson}, with a constant depending only on
$\kappa_1,\kappa_2$.

\FloatBarrier

\section*{Acknowledgments}
We thank Graham Smith for his interest in this work and for valuable comments and suggestions, especially regarding the existence of convex barriers. The first author thanks the Department of Geometry and Topology and the Institute of Mathematics (IMAG) at the University of Granada, and the Abdus Salam International Centre for Theoretical Physics (ICTP), Trieste, where part of this work was carried out, for their hospitality.

\section*{Statements and Declarations}

\subsection*{Funding}
M.P. Cavalcante received funding from the Brazilian National Council for Scientific and Technological Development (CNPq) (Grant 311136/2023-0) and Coordena\c{c}\~ao de Aperfei\c{c}oamento de Pessoal de N\'ivel Superior (CAPES), through the Math AmSud project \emph{New Trends in Geometric Analysis} (CAPES, Grant 88887.985521/2024-00).

J.M. Espinar is partially supported by the {\it Maria de Maeztu} Excellence Unit IMAG, reference CEX2020-001105-M, funded by \nolinkurl{MCINN/AEI/10.13039/501100011033/CEX2020-001105-M} and Spanish MIC Grant \nolinkurl{PID2024-160586NB-I00}. 

D.A. Marín is partially supported by the {\it Maria de Maeztu} Excellence Unit IMAG, reference CEX2020-001105-M, funded by \nolinkurl{MCINN/AEI/10.13039/501100011033/CEX2020-001105-M}, and Spanish MIC Grant PID2023-150727 NB.I00.

\subsection*{Competing interests}
The authors have no relevant financial or non-financial interests to disclose.

\subsection*{Data availability}
Data sharing is not applicable to this article as no datasets were generated
or analysed during the current study.

\subsection*{Use of AI tools}
During the preparation of this work, the authors used Claude Code (Anthropic) and 
ChatGPT Codex (OpenAI) to assist with the checking of proof arguments, with language 
editing, and with the preparation of \LaTeX{} code and figures. The authors are responsible 
for the mathematical content and for the final text.

\bibliographystyle{amsplain-initials}

\bibliography{references-final}

\providecommand{\bysame}{\leavevmode\hbox to3em{\hrulefill}\thinspace}
\providecommand{\MR}{\relax\ifhmode\unskip\space\fi MR }
\providecommand{\MRhref}[2]{%
  \href{http://www.ams.org/mathscinet-getitem?mr=#1}{#2}
}
\providecommand{\href}[2]{#2}
\begin{thebibliography}{10}

\bibitem{Ancona87}
A.~Ancona, \emph{Negatively curved manifolds, elliptic operators, and the
  {Martin} boundary}, Ann. Math. (2) \textbf{125} (1987), 495--536,
  \url{https://doi.org/10.2307/1971409}.

\bibitem{Anderson83}
M.~T. Anderson, \emph{The {Dirichlet} problem at infinity for manifolds of
  negative curvature}, J. Differ. Geom. \textbf{18} (1983), 701--721,
  \url{https://doi.org/10.4310/jdg/1214438178}.

\bibitem{AndersonSchoen85}
M.~T. Anderson and R.~Schoen, \emph{Positive harmonic functions on complete
  manifolds of negative curvature}, Ann. Math. (2) \textbf{121} (1985),
  429--461, \url{https://doi.org/10.2307/1971181}.

\bibitem{AzagraFerreraSanz2008}
D.~Azagra, J.~Ferrera, and B.~Sanz, \emph{Viscosity solutions to second order
  partial differential equations on {Riemannian} manifolds}, J. Differ. Equ.
  \textbf{245} (2008), no.~2, 307--336,
  \url{https://doi.org/10.1016/j.jde.2008.03.030}.

\bibitem{BallmannPolymerakis2020}
W.~Ballmann and P.~Polymerakis, \emph{On the differential form spectrum of
  geometrically finite orbifolds}, 2020, Preprint, arXiv:2011.13304v2 (revised
  2021), \url{https://arxiv.org/pdf/2011.13304v2}.

\bibitem{BNV1994}
H.~Berestycki, L.~Nirenberg, and S.~R.~S. Varadhan, \emph{The principal
  eigenvalue and maximum principle for second-order elliptic operators in
  general domains}, Comm. Pure Appl. Math. \textbf{47} (1994), no.~1, 47--92,
  \url{https://doi.org/10.1002/cpa.3160470105}.

\bibitem{BonorinoKlaser}
L.~P. Bonorino and P.~K. Klaser, \emph{Bounded {$\lambda$}-harmonic functions
  in domains of {$\mathbb H^n$} with asymptotic boundary with fractional
  dimension}, J. Geom. Anal. \textbf{28} (2018), no.~3, 2503--2521,
  \url{https://doi.org/10.1007/s12220-017-9915-z}. \MR{3833803}

\bibitem{Borbely}
A.~Borb\'ely, \emph{A note on the {D}irichlet problem at infinity for manifolds
  of negative curvature}, Proc. Amer. Math. Soc. \textbf{114} (1992), no.~3,
  865--872, \url{https://doi.org/10.1090/S0002-9939-1992-1069289-8}.
  \MR{1069289}

\bibitem{bourdon1996}
M.~Bourdon, \emph{Sur le birapport au bord des {${\rm CAT}(-1)$}-espaces},
  Inst. Hautes \'Etudes Sci. Publ. Math. (1996), no.~83, 95--104,
  \url{https://doi.org/10.1007/BF02698645}. \MR{1423021}

\bibitem{BridsonHaefliger}
M.~R. Bridson and A.~Haefliger, \emph{Metric spaces of non-positive curvature},
  Grundlehren der mathematischen Wissenschaften, vol. 319, Springer-Verlag,
  Berlin, 1999, \url{https://doi.org/10.1007/978-3-662-12494-9}.

\bibitem{Choi}
H.~I. Choi, \emph{Asymptotic {Dirichlet} problems for harmonic functions on
  {Riemannian} manifolds}, Trans. Am. Math. Soc. \textbf{281} (1984), 691--716,
  \url{https://doi.org/10.1090/S0002-9947-1984-0722769-4}.

\bibitem{CrandallIshiiLions1992}
M.~G. Crandall, H.~Ishii, and P.-L. Lions, \emph{User's guide to viscosity
  solutions of second order partial differential equations}, Bull. Am. Math.
  Soc., New Ser. \textbf{27} (1992), no.~1, 1--67,
  \url{https://doi.org/10.1090/S0273-0979-1992-00266-5}.

\bibitem{EO}
P.~Eberlein and B.~O'Neill, \emph{Visibility manifolds}, Pac. J. Math.
  \textbf{46} (1973), 45--109, \url{https://doi.org/10.2140/pjm.1973.46.45}.

\bibitem{EspinarMao}
J.~M. Espinar and J.~Mao, \emph{Extremal domains on {H}adamard manifolds}, J.
  Differential Equations \textbf{265} (2018), no.~6, 2671--2707,
  \url{https://doi.org/10.1016/j.jde.2018.04.044}. \MR{3804728}

\bibitem{Evans1998}
L.~C. Evans, \emph{Partial differential equations}, Graduate Studies in
  Mathematics, vol.~19, American Mathematical Society, Providence, RI, 1998.

\bibitem{GalvezLozano}
J.~A. G\'alvez and V.~Lozano, \emph{Geometric barriers for the existence of
  hypersurfaces with prescribed curvatures in
  {$\mathbb{M}^n\times\mathbb{R}$}}, Calc. Var. Partial Differential Equations
  \textbf{54} (2015), no.~2, 2407--2419,
  \url{https://doi.org/10.1007/s00526-015-0869-3}. \MR{3396457}

\bibitem{McKean}
H.~P. McKean, \emph{An upper bound to the spectrum of {$\Delta $} on a manifold
  of negative curvature}, J. Differential Geometry \textbf{4} (1970), 359--366,
  \url{https://doi.org/10.4310/jdg/1214429509}. \MR{266100}

\bibitem{QuaasSirakov2006}
A.~Quaas and B.~Sirakov, \emph{On the principal eigenvalues and the {D}irichlet
  problem for fully nonlinear operators}, C. R. Math. Acad. Sci. Paris
  \textbf{342} (2006), no.~2, 115--118,
  \url{https://doi.org/10.1016/j.crma.2005.11.003}. \MR{2193657}

\bibitem{QuaasSirakov2008}
\bysame, \emph{Principal eigenvalues and the {D}irichlet problem for fully
  nonlinear elliptic operators}, Adv. Math. \textbf{218} (2008), no.~1,
  105--135, \url{https://doi.org/10.1016/j.aim.2007.12.002}. \MR{2409410}

\bibitem{RipollTelichevesky}
J.~Ripoll and M.~Telichevesky, \emph{Regularity at infinity of {H}adamard
  manifolds with respect to some elliptic operators and applications to
  asymptotic {D}irichlet problems}, Trans. Amer. Math. Soc. \textbf{367}
  (2015), no.~3, 1523--1541,
  \url{https://doi.org/10.1090/S0002-9947-2014-06001-7}. \MR{3286491}

\bibitem{SY94}
R.~Schoen and S.-T. Yau, \emph{Lectures on differential geometry}, Conference
  Proceedings and Lecture Notes in Geometry and Topology, vol.~I, International
  Press, Cambridge, MA, 1994, Lecture notes prepared by Wei Yue Ding, Kung
  Ching Chang [Gong Qing Zhang], Jia Qing Zhong and Yi Chao Xu, Translated from
  the Chinese by Ding and S. Y. Cheng, With a preface translated from the
  Chinese by Kaising Tso. \MR{1333601}

\bibitem{Sullivan83}
D.~Sullivan, \emph{The {Dirichlet} problem at infinity for a negatively curved
  manifold}, J. Differ. Geom. \textbf{18} (1983), 723--732,
  \url{https://doi.org/10.4310/jdg/1214438179}.

\bibitem{Sullivan87}
D.~Sullivan, \emph{Related aspects of positivity in {Riemannian} geometry}, J.
  Differ. Geom. \textbf{25} (1987), 327--351,
  \url{https://doi.org/10.4310/jdg/1214440979}.

\bibitem{Yau75}
S.-T. Yau, \emph{Harmonic functions on complete {Riemannian} manifolds},
  Commun. Pure Appl. Math. \textbf{28} (1975), 201--228,
  \url{https://doi.org/10.1002/cpa.3160280203}.

\end{thebibliography}

\end{document}